\documentclass[hidelinks,reqno]{amsart-mf2}
\usepackage[a4paper,margin=2cm]{geometry}

\usepackage[defaultlines=2,all]{nowidow} 

\usepackage{ mathtools, amssymb, stmaryrd, mathdots, bbm, mleftright, faktor }
\SetSymbolFont{stmry}{bold}{U}{stmry}{m}{n} 
\usepackage{ blkarray }     

\usepackage{ letltxmacro }    
\usepackage{ xparse }
\usepackage{ stackengine, graphicx, scalerel, etoolbox }
\usepackage{xpatch} 

\usepackage[table,xcdraw,rgb]{xcolor}
\definecolor{darkazure}{HTML}{0059b3}
\definecolor{azure}{HTML}{007fff}
\definecolor{paleazure}{HTML}{a6d2ff}
\definecolor{vdarkgold}{HTML}{807100}
\definecolor{darkergold}{HTML}{b39e00}
\definecolor{darkgold}{HTML}{f2d400}
\definecolor{gold}{HTML}{ffdf00}
\definecolor{palegold}{HTML}{fff7bf}
\definecolor{beaverred}{HTML}{cc0000}
\definecolor{myred}{HTML}{f03f02}
\definecolor{palered}{HTML}{f68c67}
\definecolor{mygrey}{HTML}{e6e6e6}
\definecolor{darkgrey}{HTML}{bbbbbb}
\newcommand{\colourA}{azure}

\newcommand{\colourB}{darkergold}

\newcommand{\colourC}{beaverred}

\usepackage{ tikz, tikz-cd }
\usetikzlibrary{shapes.misc, calc, decorations.pathreplacing, patterns, arrows.meta, positioning}
\usepackage{mfabacus_edited}
\tikzset{
ablhs/.style={inner sep=0, anchor=east, font=\tiny},
abrhs/.style={inner sep=0, anchor=west, font=\tiny},
abbelow/.style={inner sep=0, anchor=center},
}
\usepackage{genyoungtabtikz}

\usepackage{ standalone }
\usepackage[section]{ placeins } 

\counterwithin{figure}{section}
\counterwithin{table}{section}

\usepackage[labelformat=simple]{ subcaption }

\usepackage[margin=0cm]{ caption }

\makeatletter
\def\subsection{\@startsection{subsection}{2}%
  \z@{.5\linespacing\@plus.7\linespacing}{.3\linespacing}%
  {\normalfont\bfseries}}
\makeatother

\makeatletter
\def\subsubsection{\@startsection{subsubsection}{3}%
  \z@{.8\linespacing\@plus.2\linespacing}{-1ex}%
  {\normalfont\itshape}}%
\makeatother

\usepackage[shortlabels]{ enumitem }
\makeatletter    
\newcommand{\prelistcommand}{\nobreak\leavevmode\@nobreaktrue}
\makeatother
\SetEnumitemKey{beginthm}{before=\prelistcommand}   
\SetEnumitemKey{smallin}{leftmargin=2.7em, labelindent=*}   
\setlist{smallin, topsep=1pt}

\setenumerate{label=(\roman*)}

\newenvironment{proofenum}{\begin{proof}\begin{enumerate}[beginthm]}{\qedhere\end{enumerate}\end{proof}}

\usepackage{ amsthm, thmtools, url }
\usepackage[pagebackref,hidelinks,linktoc=all]{hyperref}
\renewcommand*{\backrefalt}[4]{%
\ifcase #1 %
[No citations.]%
\or
[Cited on page~#2.]%
\else
[Cited on pages #2.]%
\fi
}
\usepackage[nameinlink, capitalize]{ cleveref }

\crefname{definition}{Def.}{Defs.}
\Crefname{definition}{Definition}{Definitions}

\numberwithin{equation}{section}
\declaretheorem[style=plain,numberlike=equation]{theorem}
\declaretheorem[style=plain,numberlike=theorem]{lemma}
\declaretheorem[style=plain,numberlike=theorem]{proposition}

\declaretheorem[style=remark,numberlike=theorem]{remark}

\declaretheorem[style=definition,numberlike=theorem]{definition}

\newtheorem{cor}[lemma]{Corollary}

\newcounter{equationstore}
\newcounter{subequationstore}

\AtBeginEnvironment{proof}{
    \setcounter{subequationstore}{\value{equationstore}}
    \setcounter{equationstore}{\value{equation}}
    \setcounter{equation}{0}
    \renewcommand{\theequation}{\arabic{section}.\arabic{equationstore}.\arabic{equation}}
}
\AtEndEnvironment{proof}{
    \setcounter{equation}{\value{equationstore}}
    \setcounter{equationstore}{\value{subequationstore}}
}

\crefformat{section}{\S#2#1#3}
\crefrangeformat{section}{\S#3#1#4--#5#2#6}
\crefmultiformat{section}{\S#2#1#3}{ and~\S#2#1#3}{, \S#2#1#3}{, and~\S#2#1#3}
\crefrangemultiformat{section}{\S#2#1#3}{ and~\S#2#1#3}{, \S#2#1#3}{, and~\S#2#1#3}

\crefformat{equation}{(#2#1#3)}
\crefrangeformat{equation}{(#3#1#4)--(#5#2#6)}
\crefmultiformat{equation}{(#2#1#3)}{ and~(#2#1#3)}{, (#2#1#3)}{, and~(#2#1#3)}
\crefrangemultiformat{equation}{(#2#1#3)}{ and~(#2#1#3)}{, (#2#1#3)}{, and~(#2#1#3)}

\crefname{enumi}{}{}

\NewDocumentCommand\set{s m}{%
    \IfBooleanTF#1%
    {\left\{ #2 \right\}}%
    {\{#2\}}%
}
\NewDocumentCommand\setbuild{s m m}{%
    \IfBooleanTF#1%
    {\ensuremath{\left\{\, #2 \, \middle| \, #3 \,\right\}}}%
    {\ensuremath{\{\, #2 \, \mid \, #3 \,\}}}%
}

\newcommand{\Z}{\mathbb{Z}}
\newcommand\bbz{\Z}
\newcommand{\Q}{\mathbb{Q}}
\newcommand{\R}{\mathbb{R}}

\newcommand{\N}{\mathbb{N}}

\newcommand{\al}{\alpha}
\newcommand{\be}{\beta}
\newcommand{\ze}{\zeta}
\newcommand{\si}{\sigma}
\newcommand{\ga}{\gamma}
\newcommand{\de}{\delta}

\newcommand{\ka}{\kappa}
\newcommand{\ep}{\eps}
\newcommand{\eps}{\epsilon}
\newcommand{\la}{\lambda}

\renewcommand{\epsilon}{\varepsilon}
\renewcommand{\phi}{\varphi}
\renewcommand{\emptyset}{\varnothing}

\usepackage{letltxmacro}
\makeatletter
\let\oldr@@t\r@@t
\def\r@@t#1#2{%
\setbox0=\hbox{$\oldr@@t#1{#2\,}$}\dimen0=\ht0
\advance\dimen0-0.2\ht0
\setbox2=\hbox{\vrule height\ht0 depth -\dimen0}%
{\box0\lower0.4pt\box2}}
\LetLtxMacro{\oldsqrt}{\sqrt}
\renewcommand*{\sqrt}[2][\ ]{\oldsqrt[#1]{#2}}
\makeatother

\makeatletter
\let\save@mathaccent\mathaccent
\newcommand*\if@single[3]{%
  \setbox0\hbox{${\mathaccent"0362{#1}}^H$}%
  \setbox2\hbox{${\mathaccent"0362{\kern0pt#1}}^H$}%
  \ifdim\ht0=\ht2 #3\else #2\fi
  }
\newcommand*\rel@kern[1]{\kern#1\dimexpr\macc@kerna}
\DeclareRobustCommand*\widebar[1]{\@ifnextchar^{{\wide@bar{#1}{0}}}{\wide@bar{#1}{1}}}
\newcommand*\wide@bar[2]{\if@single{#1}{\wide@bar@{#1}{#2}{1}}{\wide@bar@{#1}{#2}{2}}}
\newcommand*\wide@bar@[3]{%
  \begingroup
  \def\mathaccent##1##2{%
    \let\mathaccent\save@mathaccent
    \if#32 \let\macc@nucleus\first@char \fi
    \setbox\z@\hbox{$\macc@style{\macc@nucleus}_{}$}%
    \setbox\tw@\hbox{$\macc@style{\macc@nucleus}{}_{}$}%
    \dimen@\wd\tw@
    \advance\dimen@-\wd\z@
    \divide\dimen@ 3
    \@tempdima\wd\tw@
    \advance\@tempdima-\scriptspace
    \divide\@tempdima 10
    \advance\dimen@-\@tempdima
    \ifdim\dimen@>\z@ \dimen@0pt\fi
    \rel@kern{0.6}\kern-\dimen@
    \if#31
      \overline{\rel@kern{-0.6}\kern\dimen@\macc@nucleus\rel@kern{0.4}\kern\dimen@}%
      \advance\dimen@0.4\dimexpr\macc@kerna
      \let\final@kern#2%
      \ifdim\dimen@<\z@ \let\final@kern1\fi
      \if\final@kern1 \kern-\dimen@\fi
    \else
      \overline{\rel@kern{-0.6}\kern\dimen@#1}%
    \fi
  }%
  \macc@depth\@ne
  \let\math@bgroup\@empty \let\math@egroup\macc@set@skewchar
  \mathsurround\z@ \frozen@everymath{\mathgroup\macc@group\relax}%
  \macc@set@skewchar\relax
  \if#31
    \macc@nested@a\relax111{#1}%
  \else
    \def\gobble@till@marker##1\endmarker{}%
    \futurelet\first@char\gobble@till@marker#1\endmarker
    \ifcat\noexpand\first@char A\else
      \def\first@char{}%
    \fi
    \macc@nested@a\relax111{\first@char}%
  \fi
  \endgroup
}
\makeatother

\newcommand\ppmod[1]{\ (\operatorname{mod}\,#1)}
\newcommand{\ls}{\DOTSB\leqslant}
\newcommand{\gs}{\DOTSB\geqslant}
\renewcommand{\leq}{\DOTSB\leqslant}
\renewcommand{\geq}{\DOTSB\geqslant}
\newcommand\ip[2]{\left(#1:#2\right)} 

\newcommand\abd{abacus display\xspace}
\newcommand\abds{\abd{}s\xspace}
\newcommand\Abd{Abacus display\xspace}
\newcommand\Abds{\Abd{}s\xspace}
\newcommand{\alp}[1]{\al^{(#1)}}
\newcommand{\bep}[1]{\be^{(#1)}}

\newcommand\tcp{$2$-core\xspace}
\newcommand\tcps{\tcp{}s\xspace}

\renewcommand\iff{if and only if\xspace}

\newcommand\rb[2]{S(#1,#2)} 

\newcommand{\hatal}{\hat{\al}}

\newcommand\sspn[1]{\lan#1\ran_\ast}
\newcommand\aspn[1]{[#1]_\ast}

\newcommand\res[1]{\!\downarrow_{#1}}

\newcommand{\twoc}[1]{\ka_{#1}}
\newcommand{\threecore}[2]{\ga_{#1,#2}}
\newcommand{\thrc}[2]{\threecore{#1}{#2}}

\newcommand{\corandquot}[3]{[#1; (#2,#3)]}

\newcommand{\tbc}{\(3\)-bar-core\xspace}
\newcommand{\tbcs}{\tbc{}s\xspace}
\newcommand{\pbc}{\(p\)-bar-core\xspace}

\newcommand{\tbq}{\(3\)-bar-quotient\xspace}
\newcommand{\tbqs}{\(3\)-bar-quotient{}s\xspace}
\newcommand{\tbw}{$3$-bar-weight\xspace}
\newcommand{\pbw}{$p$-bar-weight\xspace}

\newcommand\stp{strict partition\xspace}
\newcommand\stps{\stp{}s\xspace}
\newcommand\spr{residue\xspace}
\newcommand\sprs{\spr{}s\xspace}
\newcommand\nd[1]{$#1$-node\xspace}
\newcommand\nds[1]{\nd#1{}s\xspace}
\newcommand\spre{removable\xspace}
\newcommand\sprm{\spre node\xspace}
\newcommand\sprms{\sprm{}s\xspace}
\newcommand\spad{addable\xspace}
\newcommand\spam{\spad node\xspace}
\newcommand\spams{\spam{}s\xspace}
\newcommand\esprm[1]{$#1$-removable node\xspace}
\newcommand\esprms[1]{$#1$-removable nodes\xspace}
\newcommand\espam[1]{$#1$-addable node\xspace}
\newcommand\espams[1]{$#1$-addable nodes\xspace}

\newcommand\delt[1]{\de(#1)}

\newcommand{\ydsm}{/} 
\newcommand{\rowrem}[1]{\mathrm{R}#1}
\newcommand{\colrem}[1]{\mathrm{C}#1}

\usepackage{xspace}
\newcommand{\runnerswap}[2]{S_{#1}^{(#2)}} 
\newcommand\rsf{runner-swapping function\xspace}

\newcommand{\lan}{\langle}
\newcommand{\ran}{\rangle}
\makeatletter
\newsavebox{\@brx}
\newcommand{\llangle}[1][]{\savebox{\@brx}{\(\m@th{#1\langle}\)}%
  \mathopen{\copy\@brx\kern-0.7\wd\@brx\usebox{\@brx}}}
\newcommand{\rrangle}[1][]{\savebox{\@brx}{\(\m@th{#1\rangle}\)}%
  \mathclose{\copy\@brx\kern-0.7\wd\@brx\usebox{\@brx}}}
\makeatother

\newcommand\sss{\mathfrak{S}_}
\newcommand\aaa{\mathfrak{A}_}
\newcommand\hsss{\hat{\mathfrak{S}}_}
\newcommand\haaa{\hat{\mathfrak{A}}_}
\newcommand\spnorig[1]{\lan#1\ran} 
\newcommand\spn[1]{\llangle#1\rrangle} 

\newcommand\Aspnsum[1]{\llbracket#1\rrbracket}
\newcommand\apn[1]{\Aspnsum{#1}}
\newcommand{\ass}[1]{#1^{\mathsf{a}}}
\newcommand{\conj}[1]{#1^{\mathsf{c}}}

\newcommand\br[1]{\widebar{#1}} 
\newcommand\bspn[1]{\br{\spn{#1}}}

\newcommand{\plus}{{+}}
\newcommand{\minus}{{-}}

\usepackage{mathrsfs}
\newcommand\scrd{\mathscr{D}}
\newcommand\scrp{\mathscr{P}}

\newcommand{\len}[1]{\operatorname{len}(#1)}
\newcommand\cm[2]{\left[#1:#2\right]}

\newcommand{\spindecoration}[1]{#1}
\newcommand{\netspinaddables}[2]{\spindecoration{\Delta}_{#1}#2} 
\newcommand\bswp[2]{#1^{\ast#2}} 

\usepackage{amsaddr}

\newcommand{\initialsspace}{0.1em}
 
\makeatletter
\newcommand{\splitlist}[1]{\@splitlist#1\@nil}
\def\@splitlist#1\@nil{%
  \if\relax\detokenize{#1}\relax
    \expandafter\@gobble
  \else
    \expandafter\@firstofone
  \fi
  {\@spl@tlist#1.\@nil}%
}

\def\@spl@tlist#1.#2\@nil{%
    \def\tmpA{#1}%
    \def\tmpB{#2}%
    \def\tmpP{.}%
    \ifx\tmpB\tmpP%
        #1.%
    \else{%
        \ifx\tmpA\@empty%
        \else%
                #1.\nobreak\hspace{\initialsspace}%
        \fi%
    }%
    \fi%
  \if\relax\detokenize{#2}\relax
    \expandafter\@firstoftwo
  \else
    \expandafter\@secondoftwo
  \fi
  {\unskip}%
  {\@spl@tlist#2\@nil}%
}
\makeatother

\title[Proportional spin characters in characteristic 3]
{Spin characters of symmetric and alternating groups which are proportional in characteristic~3}

\author{Matthew Fayers}
\address{\upshape \vspace{-0.2cm}%
Queen Mary University of London \endgraf
\texttt{m.fayers@qmul.ac.uk} \endgraf
}

\author{Eoghan McDowell}
\address{\upshape \vspace{-0.2cm}%
University of Bristol \endgraf
\texttt{eoghan.mcdowell@bristol.ac.uk} \endgraf
}

\usepackage{needspace}
\usepackage{changepage} 

\begin{document}

\begin{abstract}
Let $G$ be a finite group and $p$ a prime. It is interesting to determine when two ordinary irreducible representations of $G$ have the same $p$-modular reduction; this is the same as saying that the corresponding rows of the decomposition matrix are equal, or that the characters of the two representations agree on $p$-regular conjugacy classes. In fact we consider the more general problem of asking when two rows of the decomposition matrix are proportional.

In the case where $G$ is a double cover of the alternating or symmetric group, this problem has been solved except when $p=3$. 
Here we resolve the missing case for spin characters (i.e.~characters which are not lifted from the covered group), which completely solves the problem for the double cover of the symmetric group. There are surprising parallels to our solution to the corresponding problem for $p=2$.
\end{abstract}

\maketitle

\begin{adjustwidth}{70pt}{70pt}
    \tableofcontents
\end{adjustwidth}

\section{Introduction}
\thispagestyle{empty}

Suppose \(G\) is a finite group and \(p\) a prime.
One route to understanding the \(p\)-modular representation theory of \(G\) is through \emph{\(p\)-modular reduction}, in which a representation over a field of characteristic \(0\) is converted to one over a field of characteristic \(p\).
This process is well-defined at the level of composition factors.
An interesting question is to determine when two different irreducible representations have the same \(p\)-modular reduction.
This is the same as asking when the corresponding rows of the \emph{decomposition matrix} of \(G\) are equal, or when their characters have the same restrictions to the \(p\)-regular conjugacy classes of \(G\) (i.e.~they afford the same \emph{Brauer characters}).
More generally, one can ask when these rows/characters are proportional.


This paper addresses this problem for the double covers \(\hsss n\) and \(\haaa n\) of the symmetric and alternating groups.
Irreducible representations of these groups come in two types:
those that are lifted from representations of $\sss n$ or \(\aaa n\) (i.e.\ those on which the central involution $z\in\haaa n$ acts as $1$),
and the \emph{spin representations} (on which $z$ acts as $-1$).

We summarise what is known about proportionality of the corresponding Brauer characters.

\begin{itemize}
\item
The $p$-modular reductions of two non-spin characters are proportional only if they are equal. The cases where this happens for $\hsss n$, and for \(\haaa n\) when \(p=2\), were classified by Wildon \cite{wildon2008distinctrows}, and for $\haaa n$ when \(p \geq 5\) by the second author \cite{mcdowell2024charsonlprimeclasses} (the case \(p=3\) is ongoing work).
\item
A spin representation and a non-spin representation can have proportional $p$-modular reductions only if $p=2$ (for odd $p$ they lie in different $p$-blocks). The cases where this happens for $\hsss n$ were classified by the authors in \cite{fayersmcd2025spintospecht} and for $\haaa n$ by the second author in \cite{mcdowell2026spintolinearAn}.
\item
When $p\neq3$, the $p$-modular reductions of two spin characters are proportional only if they are equal, and the cases where this happens were classified by the second author in \cite{mcdowell2024charsonlprimeclasses}. In this paper we consider the exceptional case $p=3$ (where proportionality does not imply equality), and classify the pairs of spin characters with proportional $p$-modular reductions.
\end{itemize}


Our main theorems are as follows. Here $(\alp0,\alp1)$ denotes the \tbq of $\al$, defined in \cref{subsec:abacus}.
We write $\ep(\al)=1$ if $\al$ is odd (i.e.\ has an odd number of positive even parts), and $\ep(\al)=0$ otherwise.
We write $\bar\chi$ for the $3$-modular reduction of a character $\chi$.

\begin{theorem}
\label{thm:mainS}
Let \(\chi\) and \(\psi\) be distinct irreducible spin characters of \(\hsss n\), labelled by \stps \(\al\) and \(\be\) respectively.
Then $\bar\chi$ is proportional to $\bar\psi$ \iff both of the following hold:
\begin{enumerate}
    \item either $\al=\be$, or there exist \tcps \(\rho\), \(\sigma\) and \(\tau\) such that
\begin{align*}
\alp0 &= \rho + \sigma, & \bep0 &= \rho + \tau, \\
\alp1 &= \tau,          &     \bep1 &= \sigma;
\end{align*}
    \item whichever of \(\al\) and \(\be\) are odd
    (i.e.~have an odd number of even parts)
    have at least one part divisible by~\(3\).
\end{enumerate}
In this case $\bar\chi=2^{(\len{\rho+\si}-\len{\rho+\tau}-\ep(\al)+\ep(\be))/2}\bar\psi$.
\end{theorem}

\begin{theorem}
\label{thm:mainA}
Let \(\chi\) and \(\psi\) be distinct irreducible spin characters of \(\haaa n\), labelled by \stps \(\al\) and \(\be\) respectively. Then $\bar\chi$ is proportional to $\bar\psi$ \iff both of the following hold:
\begin{enumerate}
    \item either $\al=\be$, or there exist \tcps \(\rho\), \(\sigma\) and \(\tau\) such that
    \begin{align*}
\alp0 &= \rho + \sigma, & \bep0 &= \rho + \tau, \\
\alp1 &= \tau,          &     \bep1 &= \sigma;
\end{align*}
    \item whichever of \(\al\) and \(\be\) are even
    (i.e.~have an even number of even parts)
    have at least one part divisible by \(3\).
\end{enumerate}
In this case $\bar\chi=2^{(\len{\rho+\si}-\len{\rho+\tau}+\ep(\al)-\ep(\be))/2}\bar\psi$.
\end{theorem}

Note that in general there are two different proper double covers of \(\sss n\), but they are isoclinic, so it does not matter which one we choose for our results \cite[\S 6.7]{atlas}.

In the next section \Cref{sec:defs} we set out the basic theory we shall use.  
In \cref{sec:reduction_to_averaged_chars} we show that both of our theorems can be reduced to studying \emph{averaged} characters of $\hsss n$, and then in \cref{sec:idpt-of-core} we show that proportionality of characters is independent of \tbc (an analogue of a result from the \(p=2\) case).
In the remaining sections \Cref{sec:if} and \Cref{onlyifsec} we complete the proofs of the `if' and `only if' parts of our main theorems.

The `if' proof  of \Cref{sec:if} relies on knowledge of decomposition numbers in \emph{RoCK blocks} in characteristic \(3\) and (perhaps surprisingly) also in characteristic \(2\).
The `only if' proof of \Cref{onlyifsec} is inductive, using Morris's rule and the branching rule to identify pairs of proportional characters of smaller groups.

\section{Definitions and techniques}
\label{sec:defs}

\subsection{Partitions and strict partitions}
\label{subsec:partitions}

A \emph{partition} is an infinite weakly-decreasing sequence $\la=(\la_1,\la_2,\dots)$ of integers which are eventually zero. We write $|\la|$ for the sum $\la_1+\la_2+\dots$, and say that $\la$ is a partition of $|\la|$. When writing partitions, we omit trailing zeroes. We write $\emptyset$ for the unique partition of $0$. For any partition $\la$, we write $\len{\la}$ for the number of non-zero parts of $\la$, which we call the \emph{length} of $\la$. The non-zero integers $\la_1,\la_2,\dots$ are the \emph{parts} of $\la$; given a partition $\la$ and $a \in \N$, we write $a\in\la$ to mean that $a$ is a part of $\la$. A \emph{\stp} is a partition whose non-zero parts are distinct.

We write $\scrp$ for the set of all partitions, \(\scrd\) for the set of all \stps, \(\scrp(n)\) for the set of partitions of \(n\), and \(\scrd(n)\) for the set of \stps of \(n\). 

We use some natural notation for combining partitions.
Given partitions \(\la\) and \(\mu\), we write:
\begin{itemize}
\item
$a\la = (a\la_1,a\la_2,\dots)$ for \(a \in \N\);
\item
$\la+\mu = (\la_1+\mu_1,\la_2+\mu_2,\dots)$;
\item
$\la\sqcup\mu$ for the partition obtained by combining all the parts of $\la$ and $\mu$ and writing them in decreasing order.
\end{itemize}

The \emph{Young diagram} of a partition $\la$ is the set
\[
[\la]=\setbuild{(r,c)\in\N^2}{c\ls\la_r},
\]
drawn in the plane using the English convention. We call the elements of the Young diagram the \emph{nodes} of $\la$. In general, a node means an element of $\N^2$.

The \emph{conjugate} of a partition \(\la\), denoted \(\la'\), is the partition whose Young diagram is obtained by reflecting that of \(\la\) in the main diagonal.
Thus the parts of \(\la'\) are the lengths of the columns of \(\la\).

A \emph{\tcp} is a partition of the form
\[
    \twoc{r} = (r,r-1, r-2, \ldots, 2,1)
\]
for some \(r \geq 0\) (viewing \(\twoc{0} = \emptyset\)).

\subsection{Spin characters of \texorpdfstring{\(\hsss n\)}{Ŝn} and \texorpdfstring{\(\haaa n\)}{Ân}}

We let $\hsss n$ and $\haaa n$ denote double covers of $\sss n$ and $\aaa n$ respectively, so that $\hsss n$ is a Schur cover of $\sss n$ for $n\gs4$, and $\haaa n$ is the preimage of $\aaa n$ under the natural projection from $\hsss n$ to $\sss n$. Here we give some background on irreducible characters of $\hsss n$ and $\haaa n$. The group $\hsss n$ has a central involution $z$ such that $\hsss n/\lan z\ran\cong\sss n$, and similarly for $\haaa n$. We say that an element $g\in\hsss n$ \emph{projects to cycle type $\al$} if the image of $g$ under the natural map has cycle type $\al$.

Let $G=\sss n$ or $\aaa n$. If $\chi$ is an irreducible character of $\hat G$, then $\chi(z)=\pm\chi(1)$. If $\chi(z)=\chi(1)$, then $\chi$ is the inflation to $\hat G$ of a character of $G$. The irreducible characters $\chi$ for which $\chi(z)=-\chi(1)$ are called \emph{spin characters}.

The irreducible spin characters of $\hsss n$ and $\haaa n$ were classified by Schur. Before describing the classification, it will be useful to consider \emph{associate} and \emph{conjugate} characters. If $\chi$ is a character of $\hsss n$, then the associate character $\ass\chi$ is obtained by tensoring with the sign character. If $\chi$ is a character of $\haaa n$, then the conjugate character $\conj\chi$ is obtained by precomposing with conjugation by an element of $\hsss n\setminus\haaa n$.

Now let $\al\in\scrd(n)$. We say that $\al$ is \emph{even} if it has an even number of positive even parts, and \emph{odd} otherwise. If $\al$ is even, then there is an irreducible character $\lan\al\ran$ of $\hsss n$, satisfying $\ass{\lan\al\ran}=\lan\al\ran$. If $\al$ is odd, then there is a pair of irreducible characters $\lan\al\ran_+$, $\lan\al\ran_-$, satisfying $\ass{\lan\al\ran_+}=\lan\al\ran_-$.

For $\haaa n$, the opposite situation holds: if $\al$ is even, then there is a pair of irreducible characters $[\al]_+$, $[\al]_-$ of $\haaa n$, satisfying $\conj{[\al]_+}=[\al]_-$. If $\al$ is odd, then there is an irreducible character $[\al]$ satisfying $\conj{[\al]}=[\al]$.

Schur proves that these characters, listed for all \stps of $n$, give complete irredundant lists of irreducible characters for $\hsss n$ and $\haaa n$. The characters are related by restriction as follows:
\begin{equation}\label{restrictionofspin}
\lan\al\ran\res{\haaa n}=[\al]_++[\al]_-\text{ for $\al$ even},\qquad\lan\al\ran_+{}\res{\haaa n}=\lan\al\ran_-{}\res{\haaa n}=[\al]\text{ for $\al$ odd}.
\end{equation}

It will be convenient in this paper to adopt some notation introduced in \cite{fayers20spin2alt}. For any $\al\in\scrd$, we define the generalised $\hsss n$-character
\[
\spn\al=
\begin{cases}
\lan\al\ran&\text{if $\al$ is even}
\\
\frac1{\sqrt2}(\lan\al\ran_++\lan\al\ran_-)&\text{if $\al$ is odd}
\end{cases}
\]
and the generalised $\haaa n$-character
\[
\apn\al=
\begin{cases}
[\al]&\text{if $\al$ is odd}
\\
\frac1{\sqrt2}([\al]_++[\al]_-)&\text{if $\al$ is even}.
\end{cases}
\]
Using the generalised characters $\spn\al$ simplifies several results for spin representations, such as
the branching rule (\cref{spinbranch} below).
We show in \Cref{sec:reduction_to_averaged_chars} that to identify proportional pairs it suffices to consider the characters \(\spn\al\) in place of \(\spnorig{\al}\), and moreover that doing so settles the problem for \(\haaa n\) as well as for \(\hsss n\).

Given an ordinary character \(\xi\), we denote by \(\br{\xi}\) the Brauer character of the \(3\)-modular reduction of the representation affording \(\xi\).
Equivalently, \(\br{\xi}\) is obtained by restricting \(\xi\) to \(3\)-regular classes.

\subsection{Bar removal and Morris's rule}

Spin character values can be computed iteratively by removing \emph{bars} and applying \emph{Morris's rule}, analogously to removing hooks and applying the Murnaghan--Nakayama rule for linear characters.

\begin{definition}
Let $k\in \N$ odd and \(\al\in\scrd\). A \emph{$k$-bar} in $\al$ is one of the following:
\begin{itemize}
\item
a part of size equal to $k$;
\item
a pair of parts whose sizes sum to $k$;
\item
the last \(k\) nodes in a part of size $\al_r>k$ such that $\al_r-k$ is not already a part of $\al$.
\end{itemize}

\emph{Removing a \(k\)-bar} from \(\al\) means deleting the bar and, if necessary, reordering the remaining parts to form a \stp.
A \stp is a \emph{$k$-bar-core} if it has no $k$-bars.
For any $\al$, the \emph{$k$-bar-core of $\al$} is the \stp obtained by repeatedly removing $k$-bars until a $k$-bar-core is reached, and the \emph{$k$-bar-weight} of $\al$ is the number of $k$-bars removed to reach the $k$-bar-core.
\end{definition}

Morris's rule determines character values on elements projecting to cycle type with all parts odd (whilst on elements that project to cycle type with an even part, the generalised character \(\spn{\al}\) vanishes: the irreducible character(s) labelled by \(\al\) vanish on such elements unless the cycle type is equal to \(\al\) and \(\al\) is odd, in which case \(\spnorig{\al}_{\plus}\) and \(\spnorig{\al}_{\minus}\) take opposite signs \cite[Theorem~8.7]{hoffmanhumphreys1992projreps}).
For each partition \(\nu\) of \(n\) into odd parts, there are two conjugacy classes of \(\hsss n\) which project to cycle type \(\nu\), and a spin character takes values of opposite sign on these two classes. We let $g\in\hsss n$ be an element projecting to cycle type $\nu$ for which $\spn{(n)}(g)>0$, and then define \(\xi(\nu)\) to be the value of a spin character \(\xi\) on this element \(g\).

\begin{theorem}
[{Morris's rule \cite[Theorem 10.1]{hoffmanhumphreys1992projreps}}]
\label{spinmurnak}
Suppose $\al\in\scrd(n)$ and $k\in\N$ is odd. Let $\rb\al k$ be the set of \stps that can be obtained from $\al$ by removing a $k$-bar. For each $\la\in\rb\al k$ there is a non-zero constant $d_\la$ such that for any partition $\nu$ of \(n-k\) into odd parts,
\[
\spn\al(\nu\sqcup(k)) = \sum_{\la\in\rb\al k} d_\la \spn\la(\nu).
\]
\end{theorem}

We apply Morris's rule to deduce that removing bars from a proportional pair of spin characters yields another proportional pair, analogously to \cite{fayersmcd2025spintospecht}.

\begin{proposition}
[{Cf.~\cite[Lemma~3.2 and Proposition~3.3]{fayersmcd2025spintospecht}}]
\label{prop:weight_and_remove_bars_from_prop_pair}
Let \(\al,\be \in \scrd(n)\), and suppose \(\bspn{\al} \propto \bspn{\be}\). Let \(k \in \N\) be odd and not divisible by $3$. Then the \(k\)-bar-weights of \(\al\) and \(\be\) are equal. Furthermore, if \(\hat{\al}, \hat{\be}\) denote the \stps obtained by removing all \(k\)-bars from \(\al\) and \(\be\) respectively, then \(\bspn{\hat{\al}} \propto \bspn{\hat{\be}}\).
\end{proposition}

\begin{proof}
The characters \(\spn{\al}\) and \(\spn{\be}\) vanish on precisely the same \(3\)-regular classes.
In particular, the $k$-bar-weights of \(\al\) and \(\be\) are both equal to the largest $w$ such that \(\spn{\al}\) and \(\spn{\be}\) are non-zero on elements projecting to cycle type \((k^w, 1^{n-kw})\).

Now, \(\spn{\hatal}\) and \(\spn{\hat{\be}}\) both vanish on elements projecting to cycle type with an even part (by \cite[Theorem~8.7]{hoffmanhumphreys1992projreps}, as noted above).
For a partition $\nu$ of $n-wk$ into odd parts coprime to \(3\), applying \cref{spinmurnak} iteratively gives
\[
\bspn{\hatal}(\nu)
    =\frac1{d_{\hatal}} \bspn\al(\nu\sqcup(k^w))
    \propto \frac1{d_{\hatal}} \bspn\be(\nu\sqcup(k^w))
    =\frac{d_{\hat{\be}}} {d_{\hatal}} \bspn{\hat{\be}}(\nu)
\]
for some constants \(d_{\hatal}\), \(d_{\hat{\be}}\).
Thus \(\bspn{\hat{\al}} \propto \bspn{\hat{\be}}\).
\end{proof}

\subsection{Blocks and residues}

The block classification for $\hsss n$ in odd characteristic $p$ was proved by Humphreys \cite{humphreysblocks}.
It is easy to show that, in odd characteristic, a spin character can never lie in the same block as a non-spin character, so we naturally have \emph{spin blocks} and \emph{non-spin blocks}.
In this paper we are concerned with the spin blocks.
Humphreys shows that two irreducible spin characters labelled by \stps $\al$ and $\be$ lie in the same $p$-block of $\hsss n$ \iff $\al$ and $\be$ have the same \pbc, except in the case where $\al=\be$ is an odd \pbc, and the two characters are $\lan\al\ran_+$ and $\lan\al\ran_-$; in this case, these two characters each lie in a simple $p$-block. To remedy this slightly awkward situation, it is more convenient to consider \emph{superblocks}: the group algebra of $\hsss n$ can be naturally viewed as a superalgebra, and indecomposable direct summands of this superalgebra are called superblocks. Again, the superblocks are naturally divided into spin superblocks and non-spin superblocks. Each superblock is a direct sum of one or two blocks, and in fact spin superblocks are almost always the same thing as spin blocks; the only exception is that when $\al$ is an odd $p$-bar-core, the sum of the two simple blocks containing the characters $\lan\al\ran_+$ and $\lan\al\ran_-$ is a superblock. (The superalgebra approach to the representation theory of $\hsss n$ was pioneered by Kleshchev and others, and is explained in \cite{kleshbook}.)

We abuse terminology by using `$p$-block' (or simply `block') to mean `superblock' in this paper.
With this convention, it is a consequence of the above results that for any $\al\in\scrd$ the generalised character $\spn\al$ lies in a single block, and that $\spn\al$ and $\spn\be$ lie in the same block \iff $\al$ and $\be$ have the same \pbc. This necessarily means that they have the same \pbw, so we can label a block by its \pbc and \pbw (that is, by the \pbc and \pbw of any strict partition labelling a spin character in that block).

It will be helpful to have an alternative description of the blocks of $\hsss n$, using residues. For simplicity, we specialise to characteristic $3$ here.

\begin{definition}
The \emph{\spr} of a node $(r,c)$ is:
\begin{align*}
    0,& \text{ if \(c \equiv 0\text{ or }1 \ppmod{3}\);} \\
    1,& \text{ if \(c \equiv 2 \ppmod{3}\).}
\end{align*}
For $\ep\in\{0,1\}$, an \emph{\nd\ep} means a node of \spr $\ep$. The \emph{content} of a partition $\al$ is the multiset of the residues of its nodes.
\end{definition}

For example, let $\al=(7,5,4,1)$. The \spr{}s of the nodes of $\al$ are illustrated in the following diagram.
\[
\young(0100100,01001,0100,0)
\]
The content of $\al$ is $\{0^{12},1^5\}$. 

(We remark that often terms such as `spin-residue' and `spin-content' are used, particularly in papers comparing spin characters with non-spin characters. Since we only work with spin characters in this paper, we omit the prefix `spin-'.)

It was proved by Morris and Yaseen (in a slightly different formulation) that two \stps of $n$ have the same \tbc \iff they have the same content \cite[Theorem 5]{morrisyaseen}.
As a consequence, two generalised characters $\spn\al$ and $\spn\be$ lie in the same $3$-block of $\hsss n$ \iff $\al$ and $\be$ have the same content.
Thus we can label the $3$-blocks of $\hsss n$ by contents: the content of a block $B$ is the content of $\al$ for any $\spn\al$ lying in $B$.

\subsection{Induction, restriction and branching rules}\label{indressec}

A node \((r,c)\) of a \stp \(\al\) is \emph{\spre} if it can be removed, possibly together with other nodes of the same \spr, to leave a \stp.
A node $(r,c)$ which is not a node of $\al$ is \emph{\spad} if it can be added, possibly together with other nodes of the same \spr, to obtain a \stp.
An \emph{\esprm\ep} means a \spre \nd\ep, and likewise with `addable' in place of `removable'.
We denote by \(\al^{-\eps}\) the \stp obtained from \(\al\) by removing  all \esprms\eps.

For example, consider the partition \(\al = (7,5,4,1)\) depicted above.
The \esprms0 are $(1,7)$, $(3,4)$, $(3,3)$ and $(4,1)$, and \(\al^{-0} = (6,5,2)\); the unique \espam0 is $(2,6)$.

Given a (generalised) character $\chi$ for $\hsss n$, we write $\chi{}\res{\hsss{n-1}}$ for its restriction to $\hsss{n-1}$, and $\chi{}\uparrow^{\hsss{n+1}}$ for the induced character for $\hsss{n+1}$. Now suppose $\chi$ lies in a single block, and write the content of this block as $\{0^a,1^b\}$. Define $e_0\chi$ to be the component of $\chi{}\res{\hsss{n-1}}$ lying in the block with content $\{0^{a-1},1^b\}$ if there is such a block, and $e_0\chi=0$ otherwise. Define $f_0\chi$ to be the component of $\chi{}\uparrow^{\hsss{n+1}}$ lying in the block with content $\{0^{a+1},1^b\}$ if there is such a block, or $0$ otherwise. Define the characters $e_1\chi$ and $f_1\chi$ similarly.
(It follows from the branching rules below and the block classification that $\chi{}\res{\hsss{n-1}}=e_0\chi+e_1\chi$ and $\chi{}\uparrow^{\hsss{n+1}}=f_0\chi+f_1\chi$ for any $\chi$.)

Fix $\ep\in\{0,1\}$ for the rest of this subsection. The functor $e_\ep$ is defined for any $n>0$, so we can define the power $e_\ep^r$ and the divided power $e_\ep^{(r)}=e_\ep^r/r!$, and similarly for $f_\ep$.
These functors commute with reduction modulo $3$.

The branching rule for spin characters (due to Dehuai and Wybourne \cite{dehuaiwybourne}) describes the decomposition of an induced or restricted spin character into irreducibles. Combined with the block classification in terms of contents, this yields the following statement giving the effect of the functors on the characters $\spn\al$.

\begin{theorem}
[Spin branching rule]\label{spinbranch}
Suppose $\al\in\scrd(n)$ and $\la\in\scrd(n-r)$, and $\ep\in\{0,1\}$. Then $\ip{e_\ep^{(r)}\spn\al}{\spn\la}$ and $\ip{f_\ep^{(r)}\spn\la}{\spn\al}$ are non-zero if and only if $\la$ is obtained by removing $r$ \nds\ep from $\al$.
If this is the case, let $b$ be the number of values $c\geq2$ such that $\al\ydsm\la$ contains a node in column $c$ but does not contain any nodes in columns $c-1$ and $c+1$. Then
\[
\ip{e_\ep^{(r)}\spn\al}{\spn\la}=\ip{f_\ep^{(r)}\spn\la}{\spn\al}=2^{b/2}.
\]
\end{theorem}

We apply the spin branching rule to deduce that removing all \esprms\eps from a proportional pair of spin characters yields another proportional pair, analogously to \cite{fayersmcd2025spintospecht}.

\begin{proposition}
[{Cf.~\cite[Proposition~3.4]{fayersmcd2025spintospecht}}]
\label{prop:restriction_gives_prop_pair}
Let \(\al,\be\in\scrd\) with \(\bspn{\al} \propto \bspn{\be}\).
Let \(\ep \in \set{0,1}\).
Then the numbers of \esprms\eps of \(\al\) and \(\be\) are equal, and the numbers of \espams\eps of \(\al\) and \(\be\) are equal.
Furthermore, \(\bspn{\al^{-\ep}} \propto \bspn{\be^{-\eps}}\).
\end{proposition}

\begin{proof}
By \cref{spinbranch}, the number of \esprms\ep of \(\al\) is the maximal \(r\) such that $e_\ep^{(r)}\bspn\al\neq0$; since \(\bspn\al \propto \bspn\be\), this is the same as the maximal $r$ such that $e_\ep^{(r)}\bspn\be\neq0$.
Likewise for the number of \espams\eps with \(f_\eps\) in place of \(e_\eps\).

Write \(r\) for the common number of \esprms\eps in \(\al\) and \(\be\).
Then \cref{spinbranch} furthermore shows that $e_\ep^{(r)}\bspn\al$ is a non-zero multiple of $\bspn{\al^{-\ep}}$ and that $e_\ep^{(r)}\bspn\be$ is a non-zero multiple of $\bspn{\be^{-\ep}}$.
Thus
\[
\bspn{\al^{-\ep}} \propto e_\ep^{(r)}\bspn\al \propto e_\ep^{(r)}\bspn\be\propto\bspn{\be^{-\ep}}.\qedhere
\]
\end{proof}

\subsection{\texorpdfstring{\tbcs}{3-bar-cores}}

It is easy to see that a nonempty \stp is a \tbc \iff it has the form $(3l-1,3l-4,3l-7,\dots,2)$ or $(3l-2,3l-5,3l-8,\dots,1)$ for some $l\gs1$. So \tbcs are parametrised by the residue of their parts modulo \(3\) and their length: we write \(\thrc{r}{l}\) for the unique \tbc with parts congruent to \(r \in \set{1,2}\) modulo \(3\) and of length \(l \geq 1\) given by
\begin{align*}
\thrc{r}{l} &= 
    (3(l-1)+r,\ 3(l-2) + r,\ \ldots,\ 3+r,\ r) 
\end{align*}
and we write \(\thrc{2}{0} = \emptyset\) (using \(r=2\) rather than \(r=1\) for the empty \tbc means that \(\thrc{r}{l}\) always has addable nodes of residue \(2-r\)).

For example, two \tbcs are drawn as Young diagrams below, with the \spr of each node indicated. 
\Yvcentermath1
\[
\thrc{1}{5} = \young(0100100100100,0100100100,0100100,0100,0)
\quad
\thrc{2}{3} = \young(01001001,01001,01)
\]
\Yvcentermath0

\subsection{The abacus}
\label{subsec:abacus}

The abacus we use for \stps was introduced by Yates in \cite{yates}, and differs from the abacus notation used in earlier literature.
It retains the feature that a strict partition is encoded by placing beads on positions corresponding to the sizes of its parts, but additionally assigns meaning to the negative positions, and gives a convenient way to describe the core and quotient.
This abacus can be defined for all odd $p$, but for simplicity we stick to the case $p=3$.

We take an abacus with three infinite vertical runners, labelled $-1,0,1$. We mark positions on the runners corresponding to non-zero integers, with the positions on runner $i$ corresponding to the non-zero integers congruent to $i$ modulo $3$ increasing from top to bottom, with position $a+1$ directly to the right of position $a$ when $a\equiv0$ or $-1\ppmod3$. We do not mark a position $0$; we put $\times$ where position $0$ would be. So the positions are labelled as follows:
\[
\begin{tikzpicture}[scale=.9]
\foreach\x in{-1,0,1}\draw(\x,-2.5)--(\x,2.5);
\foreach\x in{-1,0,1}\draw[loosely dotted,thick](\x,2.7)--++(0,.6);
\foreach\x in{-1,0,1}\draw[loosely dotted,thick](\x,-2.7)--++(0,-.6);
\draw(-1,2)node[fill=white]{$-7$};
\draw(0,2)node[fill=white]{$-6$};
\draw(1,2)node[fill=white]{$-5$};
\draw(-1,1)node[fill=white]{$-4$};
\draw(0,1)node[fill=white]{$-3$};
\draw(1,1)node[fill=white]{$-2$};
\draw(-1,0)node[fill=white]{$-1$};
\draw(0,0)node[fill=white]{$\times$};
\draw(1,0)node[fill=white]{$1$};
\draw(-1,-1)node[fill=white]{$2$};
\draw(0,-1)node[fill=white]{$3$};
\draw(1,-1)node[fill=white]{$4$};
\draw(-1,-2)node[fill=white]{$5$};
\draw(0,-2)node[fill=white]{$6$};
\draw(1,-2)node[fill=white]{$7$};
\end{tikzpicture}
\]
Positions $3r-1,3r,3r+1$ are said to be in the $r$th \emph{row} of the abacus. We define the \emph{positive half} of a runner to be the part containing the positive positions.

Now given a \stp $\al$, the \emph{\abd} for $\al$ is obtained by placing a bead on the abacus at position $\al_r$ for every $r$, and placing a bead at position $-s$ for every $s>0$ which is \emph{not} a part of $\al$. In an \abd we say that a position is \emph{occupied} if it contains a bead, or \emph{empty} otherwise. An empty position will also be called a \emph{gap}. The construction means that position $s$ is occupied \iff position $-s$ is empty (this is why we do not define position $0$).

For example, the \abd for $(10,8,7,3,1)$ is given as follows (with all the positions above those shown taken to be occupied, and all those below empty).
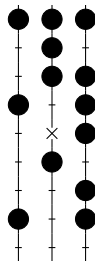
\[
\abacus(bbb,nbn,nbb,bnb,nxb,nbn,nnb,bnb,nnn)
\]

\Abds are useful for visualising the removal of $3$-bars. Removing a $3$-bar corresponds to a sequence of moves of one of two forms: sliding a pair of beads in positions $r,3-r$ up one position each into unoccupied positions $r-3,-r$; or moving a bead from position $3$ to position $-3$. As a consequence, the \abd for the \tbc of a \stp is obtained by sliding beads up their runners as far as they will go.
For the example above where \(\al=(10,8,7,3,1)\), sliding beads upwards gives
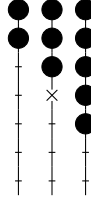
\[
\abacus(bbb,bbb,nbb,nxb,nnb,nnn,nnn)
\]
and so the \tbc of \(\al\) is $(4,1)$.

Given any \stp $\al$, we define $\delt\al$ to be the number of parts of $\al$ congruent to $1 \ppmod{3}$ minus the number of parts congruent to $2 \ppmod{3}$; in terms of the abacus, \(\delt\al\) is the number of beads in positive positions on runner~$1$ of the \abd for $\al$ minus the number of beads in positive positions on runner~$-1$. 
It is easy to see (either from the abacus or directly from the definition of $3$-bars) that if $\ga$ is the \tbc of $\al$ then $\delt\ga=\delt\al$.
Moreover, different \tbcs give different values of $\delt\ga$ (specifically, \(\delt{\thrc{1}{l}} = l\) and \(\delt{\thrc{2}{l}} = -l\)), so $\delt\al$ determines the \tbc of $\al$.

We define the \emph{largest} runner on the \abd for $\al$ to be runner $1$ if $\delt\al>0$ (i.e.~if \(\al\) has \tbc \(\thrc{1}{l}\) for some \(l\geq1\)), or runner $-1$ if $\delt\al\ls0$ (i.e.~if \(\al\) has \tbc \(\thrc{2}{l}\) for some \(l\geq0\)).
The runner opposite the largest runner is called the \emph{smallest} runner.

Now we can define the \emph{\tbq} of a \stp $\al$. This is an ordered pair $(\alp0,\alp1)$ of partitions, with $\alp0$ being strict. Construct the \abd for $\al$ as explained above.
Then:
\begin{itemize}
    \item $\alp0$ is defined by listing the positive occupied positions on runner $0$ (that is, the parts of $\al$ divisible by $3$) and dividing them by $3$;
    \item $\alp1$ is defined by examining the largest runner and setting $\alp1_r$ to be the number of gaps above the $r$th lowest bead on this runner (that is, we read the largest runner as an abacus in the sense commonly used for partitions, as described for example in \cite{jameskerber1984reptheory}).
\end{itemize}
For example, from the \abd above for $\al=(10,8,7,3,1)$, we see that $\alp0=(1)$, while $\alp1=(2,2,1,1,1)$ (with runner $1$ being largest).

It is easy to see that $\al$ is determined by its \tbc and \tbq: to construct the \abd for $\al$, we take the \abd for the \tbc of $\al$, and then slide beads down on the largest runner in accordance with $\alp1$, with corresponding moves on the smallest runner made so that position $r$ is occupied \iff position $-r$ is empty. On runner $0$ the beads in positive positions (and the corresponding gaps in negative positions) correspond to the parts of $\alp0$. Moreover, this construction shows that for any \tbc $\ga$, any \stp $\alp0$ and any partition $\alp1$, there is a \stp $\al$ with \tbc $\ga$ and \tbq $(\alp0,\alp1)$. We write this \stp as $\corandquot\ga{\alp0}{\alp1}$. From the above description of the effect of removing $3$-bars on the abacus we see that the \tbw of $\corandquot\ga{\alp0}{\alp1}$ is $|\alp0|+|\alp1|$.

\Abds are also useful for visualising addition and removal of nodes. The next two lemmas follow easily from the definitions.

\begin{lemma}
\label{lemma:node_removal_bead_move_correspondence}
Suppose $\al,\be\in\scrd$, and that $\be$ is obtained from $\al$ by removing a $0$-node.
Then the \abd for $\be$ is obtained from the \abd for $\al$ either by simultaneously moving two beads on runners $0$ and $1$ one place to the left (from positions \(\al_r\) and \(1-\al_r\) to positions \(\al_r-1\) and \(-\al_r\)), or by moving a bead from position $1$ to position $-1$.
Hence the \abd for $\al^{-0}$ is obtained by moving all beads as far to the left as possible.
\end{lemma}

\begin{lemma}
\label{lemma:pos_of_gap_and_no_of_beads}
Suppose \(\al\in\scrd\) has nonempty \tbc \(\thrc{r}{l}\).
Then, in the \abd for \(\al\), the first gap on the largest runner is in row \(l-\len{\alp1}+r-1\).
\end{lemma}

We end this section by using the abacus to describe \stps whose \tbc is relatively large.

\begin{lemma}
\label{lemma:weakly_3-separated-form}
Let \(\al \in \scrd\) with \tbc \(\thrc{r}{l}\) and \tbq \((\alp0, \alp1)\). Suppose \(l\geq \len{\alp1}\).
Then:
\begin{enumerate}[(i)]
\item\label{weaksep}
\(
    \al = (\thrc{r}{l} + 3\alp1) \sqcup 3\alp0
\);
\item\label{largebar}
if \(l \geq 1\), then \(\al\) has a unique largest bar of length congruent to \(r\) modulo $3$, given by the sum of the largest part of \(\al\) divisible by $3$ 
and the largest part of \(\al\) not divisible by $3$.
\end{enumerate}
\end{lemma}

\begin{proof}
\begin{enumerate}[beginthm,label=(\roman*)]
\item
We give the proof for the case \(r=1\); the case \(r=2\) is similar (with runners \(1\) and \(-1\) swapped).
The parts of \(\al\) divisible by $3$ are encoded precisely as \(3\alp0\) for any \stp \(\al\).
Meanwhile, \(\alp1\) is encoded on runner \(1\) with exactly \(l\) beads on the positive half and no gaps in the negative half (because \(l \geq \len{\alp1}\)).
This arrangement determines the \abd for \(\al = \corandquot{\thrc{1}{l}}{\alp0}{\alp1}\), having no beads on the positive half of runner~\(-1\). Moreover, on the positive half of runner~\(1\), the \(i\)th lowest bead must lie in row \(l-i+\alp1_i\) (because above the \(i\)th bead there must be \(l-i\) beads and \(\alp1_i\) gaps). Hence this bead corresponds to a part of size \(3(l-i+\alp1_i)+1 = (\thrc{1}{l})_i + 3\alp1_i\) as required.
\item
The expression from part \Cref{weaksep} implies that every part of $\al$ is congruent to either $0$ or $r$ modulo $3$, and the assumption \(l\geq 1\) guarantees that there is a at least one part congruent to $r$.
So a bar given by the sum of two parts of $\al$ has length $r \ppmod{3}$ \iff one of the two parts is divisible by $3$ and the other is not. Clearly taking the two largest such parts gives the largest such bar, and this is obviously larger than any bar which is a subset of a single part.
\qedhere
\end{enumerate}
\end{proof}

\section{Reduction to averaged characters of \texorpdfstring{\(\hsss n\)}{the double cover of Sn}}
\label{sec:reduction_to_averaged_chars}

The goal of this section is to reduce our main theorems (\Cref{thm:mainS,thm:mainA}) to the following statement about averaged characters of \(\hsss n\). 

\begin{theorem}
\label{thm:main_reduced}
Suppose $\al,\be\in\scrd(n)$ are distinct. Then \(\spn{\al}\) and \(\spn{\be}\) are proportional in characteristic \(3\) if and only if \(\al\) and \(\be\) have the same \tbc and there exist \tcps \(\rho\), \(\sigma\) and \(\tau\) such that:
\begin{align*}
\alp0 &= \rho + \sigma, & \bep0 &= \rho + \tau, \\
\alp1 &= \tau,          &     \bep1 &= \sigma.
\end{align*}
In this case $\br{\spn\al}=2^{(\len{\rho+\si}-\len{\rho+\tau})/2}\br{\spn\be}$.
\end{theorem}

We arrive at this reduction by considering the differences between pairs of associate or conjugate characters. We begin by noting the following facts about such characters from \cite[Theorem 8.7]{hoffmanhumphreys1992projreps}.
Given a strict partition \(\al\), we write \(\sspn\al\) for an irreducible character of \(\hsss n\) labelled by \(\al\) (that is, \(\sspn\al = \lan\al\ran\) if \(\al\) is even and \(\sspn\al=\lan\al\ran_\pm\) if \(\al\) is odd), and similarly write \(\aspn\al\) for an irreducible character of \(\haaa n\) labelled by \(\al\).

\newlength\algn\settowidth\algn{ if \(g\) projects to cycle type \(\al\) (in which case \(\al\) is even);}

\begin{lemma}
\label{lemma:char_values}
Let \(\al\in\scrd(n)\).
Let $\sspn\al$ be an irreducible character of $\hsss n$
and $\aspn\al$ an irreducible character of~$\haaa n$.
\begin{enumerate}[itemsep=1pt]
\item\label{item:values_of_associate_pairs}
    Let \(g \in \hsss n\).
    Then
    \begin{align*}
\sspn{\al}(g) &=
    \begin{cases}
        -\ass{\sspn{\al}}(g) \neq 0 & \text{ if \(g\) projects to cycle type \(\al\) and \(\al\) is odd;} \\
        \ass{\sspn{\al}}(g) & \text{ otherwise.}
    \end{cases}
\intertext{\item\csname cref@label\endcsname{item:values_of_conjugate_pairs} 
    Let \(g \in \haaa n\).
    Then}
\aspn{\al}(g) &=
    \begin{cases}
        \conj{\aspn{\al}}(g) + x 
            & \text{ if \(g\) projects to cycle type \(\al\) (in which case \(\al\) is even),} \\
       \conj{\aspn{\al}}(g) & \text{ otherwise,}
    \end{cases}
\intertext{for some \(x \not\in \Q\) (depending on \(\al\)).
\item\csname cref@label\endcsname{item:chars_on_odd_perms}
    Let \(g \in \hsss n \setminus \haaa n\).
    Then}
    \sspn{\al}(g) &=
    \begin{cases}
        - \ass{\sspn{\al}}(g) \neq 0 & \text{ if \(g\) projects to cycle type \(\al\) (in which case \(\al\) is odd);} \\
       0 & \text{ otherwise.}
    \end{cases}
    \end{align*}
    In any case, \(\spn{\al}(g) = 0\).
\end{enumerate}
\end{lemma}

These statements allow us to deduce the following.

\begin{cor}\label{sameptn}
\begin{enumerate}[beginthm]
\item
Suppose $\al\in\scrd$ is odd. Then the following are equivalent:
\begin{itemize}
\item
$\br{\lan\al\ran_+}$ is proportional to $\br{\lan\al\ran_-}$;
\item
$\br{\lan\al\ran_+}$ is equal to $\br{\lan\al\ran_-}$;
\item
$\al$ has a part divisible by $3$.
\end{itemize}
\item
Suppose $\al\in\scrd$ is even. Then the following are equivalent:
\begin{itemize}
\item
$\br{[\al]_+}$ is proportional to $\br{[\al]_-}$;
\item
$\br{[\al]_+}$ is equal to $\br{[\al]_-}$;
\item
$\al$ has a part divisible by $3$.
\end{itemize}
\end{enumerate}
\end{cor}

\begin{proof}
We prove only (i); the proof of (ii) is similar (using \Cref{lemma:char_values}\Cref{item:values_of_conjugate_pairs} in place of \Cref{lemma:char_values}\Cref{item:values_of_associate_pairs}).

By \Cref{lemma:char_values}\Cref{item:values_of_associate_pairs}, $\lan\al\ran_+$ and $\lan\al\ran_-$ differ only on elements of $\hsss n$ that project to cycle type $\al$. In particular, $\lan\al\ran_+$ and $\lan\al\ran_-$ have the same degree, so $\br{\lan\al\ran_+}$ and $\br{\lan\al\ran_-}$ can be proportional only if they are equal. This happens \iff the elements on which $\lan\al\ran_+$ and $\lan\al\ran_-$ differ (i.e.\ those that project to cycle type $\al$) are $3$-singular.
If $g\in\hsss n$ projects to an element $h\in\sss n$ of cycle type $\al$, then the order of $g$ is either equal to the order of $h$ or twice as large, and so the order of \(g\) is divisible by $3$ \iff $\al$ has a part divisible by $3$.
\end{proof}

\begin{cor}\label{sumsprop}
Suppose $\al,\be\in\scrd(n)$. Then $\br{\spn\al}$ is proportional to $\br{\spn\be}$ \iff $\br{\apn\al}$ is proportional to~$\br{\apn\be}$, with the same constants of proportionality.
\end{cor}

\begin{proof}
Using \cref{restrictionofspin} and \cref{lemma:char_values}\ref{item:chars_on_odd_perms}, we can write
\[
\spn\al(g)=
\begin{cases}
\apn\al(g)&\text{if }g\in\haaa n,
\\
0&\text{if }g\in\hsss n\setminus\haaa n,
\end{cases}
\]
and similarly for $\be$. The result follows immediately.
\end{proof}

\begin{lemma}
\label{lemma:prop_to_associates_or_conjugates}
Let \(\al,\be\in\scrd(n)\).
\begin{enumerate}
\item
Suppose $\sspn\al$ and $\sspn\be$ are distinct irreducible spin characters of $\hsss n$ labelled by $\al,\be$. If \(\br{\sspn\al}\) is proportional to \(\br{\sspn\be}\), then \(\br{\sspn\al}=\br{\ass{\sspn\al}}\) and \(\br{\sspn\be}=\br{\ass{\sspn\be}}\).
\item
Suppose $\aspn\al$ and $\aspn\be$ are distinct irreducible spin characters of $\haaa n$ labelled by $\al,\be$. If \(\br{\aspn\al}\) is proportional to \(\br{\aspn\be}\), then \(\br{\aspn\al}=\br{\conj{\aspn\al}}\) and \(\br{\aspn\be}=\br{\conj{\aspn\be}}\).
\end{enumerate}
\end{lemma}

\begin{proofenum}
\item
First we consider the case $\al=\be$. In this case, in order for $\sspn\al$ and $\sspn\be$ to be distinct, $\al$ must be odd, and the two characters are $\lan\al\ran_+$ and $\lan\al\ran_-$ in some order. So the result follows from \cref{sameptn}(i).

Now suppose $\al\neq\be$, and suppose, towards a contradiction, that \(\br{\sspn\be}\neq\br{\ass{\sspn\be}}\). Then in particular \(\sspn\be\neq\ass{\sspn\be}\), so that $\be$ is odd, and by \cref{sameptn} $\be$ does not have a part divisible by $3$. Let \(g \in \hsss n\) be an element that projects to cycle type $\be$; then $g$ is a $3'$-element of $\hsss n$. Furthermore, $g\notin\haaa n$ because $\be$ is odd, so (by \Cref{lemma:char_values}\Cref{item:chars_on_odd_perms}) $\sspn\al(g)=0$, while $\sspn\be(g)\neq0$, contradicting the assumption that \(\br{\sspn\al}\) and \(\br{\sspn\be}\) are proportional.
\item
The case $\al=\be$ is treated in the same way as in (i).
Now suppose $\al\neq\be$, and suppose, towards a contradiction, that \(\br{\aspn\be}\neq\br{\ass{\aspn\be}}\). Then in particular \(\aspn\be\neq\ass{\aspn\be}\), so that $\be$ is even, and by \cref{sameptn} $\be$ does not have a part divisible by $3$. Let \(g \in \haaa n\) be an element that projects to cycle type $\be$, and let $t\in\hsss n\setminus\haaa n$, so that $g$ and $g^t$ are $3'$-elements lying in different conjugacy classes of $\haaa n$; then \cref{lemma:char_values}\ref{item:values_of_conjugate_pairs} gives $\aspn\al(g^t)=\conj{\aspn\al}(g)=\aspn\al(g)$ and $\aspn\be(g^t)=\conj{\aspn\be}(g)\neq\aspn\be(g)$, contradicting the assumption that \(\br{\aspn\al}\) and \(\br{\aspn\be}\) are proportional.
\end{proofenum}

Recall that \(\eps(\al) = 1\) if \(\al\) is odd and \(\eps(\al) = 0\) otherwise.

\begin{proposition}
\label{prop:reduction_to_averaged_chars}
Let \(\al,\be\in\scrd(n)\) and \(c \in \R\).
\begin{enumerate}
    \item
    \label{item:snreduction}
    Let $\sspn\al$ and $\sspn\be$ be distinct irreducible characters of $\hsss n$.
    Then \(\br{\sspn\al} = c \br{\sspn\be}\)
    if and only if
    \begin{itemize}
        \item
        \(\br{\spn{\al}}=2^{(\eps(\al)-\eps(\be))/2} c \br{\spn{\be}}\),
        and
        \item  whichever of \(\al\) and \(\be\) are odd have a part divisible by \(3\).
    \end{itemize}
    \item
    \label{item:anreduction}
    Let $\aspn\al$ and $\aspn\be$ be distinct irreducible characters of $\haaa n$.
    Then \(\br{\aspn\al} = c \br{\aspn\be}\)
    if and only if
    \begin{itemize}
        \item
        \(\br{\spn{\al}} = 2^{(\eps(\be) - \eps(\al))/2} c \br{\spn{\be}}\),
        and
        \item whichever of \(\al\) and \(\be\) are even have a part divisible by \(3\).
    \end{itemize}
\end{enumerate}
\end{proposition}

\begin{proofenum}
\item
Suppose \(\br{\sspn\al} = c \br{\sspn\be}\).
Then \cref{lemma:prop_to_associates_or_conjugates}(i) gives $\br{\sspn\al}=\br{\ass{\sspn\al}}$, and so $\br{\spn\al} = 2^{\eps(\al)/2} \br{\sspn\al}$;
similarly $\br{\spn\be} = 2^{\eps(\be)/2} \br{\sspn\be}$.
Thus $\br{\spn\al} = 2^{(\eps(\al) - \eps(\be))/2} c \br{\spn\be}$, and the second statement follows from \cref{sameptn}(i).

Conversely, suppose $\br{\spn\al} = 2^{(\eps(\al) - \eps(\be))/2} c \br{\spn\be}$, and that whichever of \(\al\) and \(\be\) are odd have a part divisible by \(3\).
Then (using \cref{sameptn}(i) in the case where $\al$ is odd) we have $\br{\spn\al} = 2^{\eps(\al)/2}\br{\sspn\al}$,
and similarly for $\be$,
so that $\br{\sspn\al} = c \br{\sspn\be}$.

\item
Suppose \(\br{\aspn\al} = c \br{\aspn\be}\).
Then \cref{lemma:prop_to_associates_or_conjugates}(ii) gives $\br{\aspn\al}=\br{\conj{\aspn\al}}$, and so $\br{\apn\al} = 2^{(1-\eps(\al))/2} \br{\aspn\al}$; 
similarly $\br{\apn\be} = 2^{(1-\eps(\be))/2} \br{\aspn\be}$.
Thus $\br{\apn\al} = 2^{(\eps(\be) - \eps(\al))/2} c \br{\apn\be}$, and therefore (by \cref{sumsprop}) $\br{\spn\al} = 2^{(\eps(\be) - \eps(\al))/2} c \br{\spn\be}$.
The second statement follows from \cref{sameptn}(ii).

Conversely, suppose $\br{\spn\al} = 2^{(\eps(\be) - \eps(\al))/2} c \br{\spn\be}$, and that whichever of \(\al\) and \(\be\) are even have a part divisible by \(3\).
Then $\br{\apn\al} = 2^{(\eps(\be) - \eps(\al))/2} c \br{\apn\be}$ by \cref{sumsprop}.
Furthermore (using \cref{sameptn}(ii) in the case where $\al$ is even) we have $\br{\aspn\al} = 2^{(1 - \eps(\al))/2} \br{\apn\al}$, and similarly for $\be$. So $\br{\aspn\al} = c\br{\aspn\be}$.
\end{proofenum}

Given \Cref{prop:reduction_to_averaged_chars}, it is clear that \Cref{thm:main_reduced} implies \Cref{thm:mainS,thm:mainA}. We therefore focus on the generalised characters $\spn\al$ for the remainder of the paper, with the goal of proving \Cref{thm:main_reduced}.

\begin{remark}
The constants of proportionality in \Cref{thm:main_reduced} can be deduced from the \emph{regularisation theorem} for spin characters. In characteristic \(3\) this was first proved by Bessenrodt, Morris and Olsson \cite[Theorem~4.1]{bessenrodt1994spinchar3}. A version for all odd primes (with a different labelling of irreducible characters) was proved by Brundan and Kleshchev \cite[Theorem~1.2(i)]{brundankleshchev06regularisation} (a more useful statement for our purposes is given in \cite[Theorem~3.1]{fayers-morotti}). The regularisation theorem describes the label and multiplicity of the most-dominant (or least-dominant, depending on which labelling is used) composition factor of a \(p\)-modular reduction, and if two Brauer characters are proportional, then these composition factors must be the same and the ratio of the multiplicities is the constant of proportionality.
The description of the label furthermore provides a restriction on which characters can be proportional, and this restriction was a key part of our proof in \cite{fayersmcd2025spintospecht} that certain characters are not proportional in characteristic \(2\).
However, in this paper, we do not need to use this restriction to rule out proportionality in characteristic \(3\) (see \Cref{onlyifsec}), and our proof of proportionality (see \Cref{sec:if}) yields the constants immediately, so we do not need to appeal to the regularisation theorem.
\end{remark}

\section{Proportionality is independent of core}
\label{sec:idpt-of-core}

A pair of proportional characters must lie in the same \(3\)-block, and hence the labelling partitions have the same \tbc.
The goal of this section is to show that, aside from this requirement, a pair of spin characters being proportional does not depend on the \tbc, and instead depends only on the \tbqs.
That is, we prove the following \lcnamecref{prop:independent_of_core}.

\begin{proposition}
\label{prop:independent_of_core}
Suppose \(\alp0,\bep0\in\scrd\) and \(\alp1,\bep1\in\scrp\). If there exists a \tbc \(\gamma\) such that \(\bspn{\corandquot{\gamma}{\alp0}{\alp1}}\) is proportional to \(\bspn{\corandquot{\gamma}{\bep0}{\bep1}}\), then \(\bspn{\corandquot{\gamma}{\alp0}{\alp1}}\) is proportional to \(\bspn{\corandquot{\gamma}{\bep0}{\bep1}}\) for every \tbc \(\gamma\), with the same constant of proportionality for each $\ga$.
\end{proposition}

\noindent
This will be key to our proof of our (reduced) main theorem \Cref{thm:main_reduced}, in both directions: we will show proportionality holds in RoCK blocks, and deduce that it holds for all blocks;
meanwhile, considering examples of proportionality with various different \tbcs allows us to deduce different constraints on the \tbqs.

Recall that for $r\in\{1,2\}$ and $l\gs2-r$ we write
\begin{align*}
\thrc{r}{l} &= 
    (3(l-1)+r,\ 3(l-2) + r,\ \ldots,\ 3+r,\ r). 
\end{align*}
We prove \Cref{prop:independent_of_core} using the `runner-swapping function' introduced by the authors in \cite{fayersmcd2025spintospecht}, here with \sprs interpreted modulo \(3\). Recall the functors $e_\ep,f_\ep$ introduced in \cref{indressec}.

\begin{definition}[Runner-swapping function, {\cite[Definition~5.4]{fayersmcd2025spintospecht}}]
\label{def:runnerswapping}
For \(c\in\bbz\) and \(\eps \in \set{0,1}\), define the \emph{\rsf} to be 
\[
\runnerswap{\eps}{c} =
    \sum_{a\geq \max\{0,-c\}} (-1)^{a+c} f_\eps^{(a+c)} e_{\eps}^{(a)}.
\]
\end{definition}

\begin{remark}
After the publication of \cite{fayersmcd2025spintospecht}, Markus Linckelmann pointed out to us that the runner-swapping function has the appearance of a \emph{perfect isometry} (as introduced by Brou\'{e} \cite{broue1990perfectisometries}).
We explain here that this is indeed the case for (non-spin) blocks of the symmetric group.
Certainly our map induces a signed bijection on characters \cite[Theorem~6.1]{fayersmcd2025spintospecht}. 
We use \cite[Th\'{e}or\`{e}me~11]{enguehard1990perfectisometries} to deduce that this isometry is perfect: in the notation of \cite[p.~165]{enguehard1990perfectisometries}, our sign map is \(\al(\la) = (-1)^{|A|}\) where \(A\) is the set of \(\eps\)-addable nodes of \(\la\) (in the non-spin sense); condition (b) follows from the observation that swapping the \(\eps\)-runner with an adjacent runner changes the leg length of a hook if and only if removing that hook changes the number of \(\eps\)-addable nodes, as can be verified by considering the possible abacus configurations for the hook and using the fact that the leg length is the number of beads passed when the hook is removed.
These blocks were known to be perfectly isometric \cite{enguehard1990perfectisometries} and moreover derived equivalent \cite{chuangrouquier2008sl2}.
For more on perfect isometries, see for example \cite{brunatgramain2017perfectisometries} or \cite[Chapter~9]{linckelmann2018blocktheory}.
\end{remark}

We aim to show that the runner-swapping function has the following effect on spin characters. (Note that we only ever use the function $\runnerswap0c$ for odd values of $c$.)

\begin{proposition}
\label{prop:runner-swap_action_on_spin_chars}
Let \((\alp0,\alp1)\in\scrd\times\scrp\), and let \(r \in \set{1,2}\) and $l\gs1$. Then:
\begin{align*}
    \runnerswap{0}{2l-1} \spn{ \corandquot{\thrc{2}{l-1}}{\alp0}{\alp1} }
     &= \pm \spn{ \corandquot{\thrc1l}{\alp0}{\alp1} },\quad&    \runnerswap{0}{1-2l} \spn{ \corandquot{\thrc{1}{l}}{\alp0}{\alp1} }
     &= \pm \spn{ \corandquot{\thrc{2}{l-1}}{\alp0}{\alp1} }, \\
    \runnerswap{1}{l} \spn{ \corandquot{\thrc{1}{l}}{\alp0}{\alp1} }
     &= \pm \spn{ \corandquot{\thrc{2}{l}}{\alp0}{\alp1} },\quad&
    \runnerswap{1}{-l} \spn{ \corandquot{\thrc{2}{l}}{\alp0}{\alp1} }
     &= \pm \spn{ \corandquot{\thrc{1}{l}}{\alp0}{\alp1} }.
\end{align*}
\end{proposition}

\noindent
This result tells us that we can use runner-swapping functions to grow or shrink the \tbc of a partition without changing the \tbq.

Given \Cref{prop:runner-swap_action_on_spin_chars}, it is straightforward to deduce \Cref{prop:independent_of_core}, since the \(\eps\)-induction and -restriction functors commute with \(3\)-modular reduction, and any \tbc can be reached from any other by the action described in \Cref{prop:runner-swap_action_on_spin_chars}.
Therefore the purpose of the remainder of this section is to prove \Cref{prop:runner-swap_action_on_spin_chars}.

To prove \Cref{prop:runner-swap_action_on_spin_chars}, we first construct a \stp $\bswp\al\ep$ as follows.
\begin{description}

\item[\hspace{-\leftmargin}If \(\eps=0\)] 
examine each pair of adjacent columns with \spr $0$, indexed by \(d\) and \(d+1\) for some non-negative \(d\equiv0\ppmod3\), in turn: if $\al$ has \spams but not \sprms in columns \(d\) and \(d+1\), 
then we add these \spams; if $\al$ has \sprms but not \spams in columns \(d\) and \(d+1\), 
then we remove the \sprms; otherwise we do nothing in columns \(d\) and \(d+1\).
(If \(d=0\), we ignore all references to column \(d\).)

\item[\hspace{-\leftmargin}If \(\eps=1\)]
examine each column with \spr~$1$, indexed by some non-negative \(d \equiv 2 \ppmod{3}\), in turn:
if \(\al\) has an \spam in column \(d\), then add this \spam;
if \(\al\) has a \sprm in column \(d\), remove this \sprm; 
otherwise we do nothing in column \(d\).
\end{description}

Now the proof of \Cref{prop:runner-swap_action_on_spin_chars} rests on the following \lcnamecref{eq:runner-swap_action_on_spin_char_inexplicit}, which is the analogue of \cite[Proposition 6.10]{fayersmcd2025spintospecht}.
(Note that in comparison with \cite[\S 6]{fayersmcd2025spintospecht}, here we have dropped the decorations on symbols that indicate the spin version of various operations, because here we use only the spin versions.)

\begin{proposition}
\label{eq:runner-swap_action_on_spin_char_inexplicit}
If \(\netspinaddables{\eps}{\al}\) denotes the number of \espams\eps minus the number of \esprms\eps in a \stp \(\al\), then
\[
\runnerswap{\eps}{\netspinaddables{\eps}{\al}} \spn{\al} = \pm \spn{\bswp\al\ep}.
\]
\end{proposition}

\begin{proof}[Proof (sketch)]
The proof of \Cref{eq:runner-swap_action_on_spin_char_inexplicit} follows the same reasoning as in \cite[\S6]{fayersmcd2025spintospecht}; we sketch the proof here, but omit the details.
First, analogously to \cite[Lemma~6.2 and Lemma~6.9]{fayersmcd2025spintospecht}, we show that spin characters appearing with non-zero coefficient in  \(\runnerswap{\eps}{c} \spn{\al}\) must be labelled by partitions obtained from \(\al\) by removing all \esprms\eps which are not adjacent to an \espam\eps (and then possibly adding some \espams\eps): if such a node is not removed, then there are contributions to the alternating sum given by removing-then-adding-back and by doing nothing, and these contributions cancel out. 
Next, analogously to \cite[Proposition~6.3 and Proposition~6.10]{fayersmcd2025spintospecht}, we show that choosing \(c=\netspinaddables{\eps}{\al}\) forces adding all \espams\eps which are not adjacent to an \esprm\eps, which precisely yields \(\bswp{\al}{\eps}\), and we calculate that the coefficient of \(\spn{\bswp{\al}{\eps}}\) is \(\pm1\).
The arguments for \(\eps=0\) follow exactly as in \cite[Lemma~6.9 and Proposition~6.10]{fayersmcd2025spintospecht} by considering the pairs of adjacent columns of \spr \(0\); the arguments for \(\eps=1\) follow closely the non-spin case \cite[Lemma~6.2 and Proposition~6.3]{fayersmcd2025spintospecht} with `diagonals' replaced by `columns', since columns of \spr \(1\) come in singletons rather than pairs.
\end{proof}

Given \cref{eq:runner-swap_action_on_spin_char_inexplicit}, we need to describe \(\bswp{\al}{\eps}\) in terms of its \tbc and \tbq. We use an observation analogous to \cite[Lemma~6.11]{fayersmcd2025spintospecht}.

\begin{lemma}\label{describebsw}
Suppose $\al\in\scrd$.
\begin{enumerate}
\item
$\bswp\al0$ is obtained from $\al$ by simultaneously:
\begin{itemize}
\item
replacing each part $d\equiv1\ppmod3$ (\(d \gs4\)) with \(d-2\);
\item
replacing each part $d\equiv2\ppmod3$ with \(d+2\);
\item
inserting the part $1$, if $1\notin\al$;
\item
removing the part $1$, if $1\in\al$;
\end{itemize}
and then reordering parts into decreasing order.
\item
$\bswp\al1$ is obtained from $\al$ by simultaneously:
\begin{itemize}
\item
replacing each part $d\equiv1\ppmod3$ with \(d+1\);
\item
replacing each part $d\equiv2\ppmod3$ with \(d-1\);
\end{itemize}
and then reordering parts into decreasing order.
\end{enumerate}
\end{lemma}

\begin{proofenum}
\item
We can consider each pair of columns \(d,d+1\) with \(d \equiv 0 \ppmod{3}\) of \spr \(0\) (ignoring references to column \(d\) when \(d=0\)) in isolation and verify the result. There are eight cases to check, depending on which of the integers $d-1,d,d+1$ are parts of $\al$.
\item
We can consider each column \(d\) with \(d \equiv 2\ppmod{3}\) of \spr \(1\) in isolation and verify the result. There are four cases to check, depending on which of the integers \(d-1,d\) are parts of \(\al\).
\end{proofenum}

We can then describe \(\bswp{\al}{\eps}\) on the \(3\)-bar-abacus, and hence in terms of \tbc and \tbq.

\begin{proposition}
\label{prop:runner-swap_swaps_runners}
Let \(\eps \in \set{0,1}\).
\begin{enumerate}
    \item
    \label{item:runner-swap_swaps_runners}
Let \(\al \in \scrd\). The \(3\)-bar-\abd for \(\bswp{\al}{\eps}\) is obtained by swapping runners \(-1\) and \(1\) in the \(3\)-bar-\abd for \(\al\), and then, if \(\eps=1\), shifting the new runner \(-1\) down (i.e.~numerically increasing the beads) one position and shifting runner \(1\) up (i.e.~numerically decreasing the beads) one position.
    \item 
    \label{item:3-bar-quotient-unchanged}
Let \(\ga\) be a \tbc and \((\alp0,\alp1)\in\scrd\times\scrp\). Then
\[
    \bswp{\corandquot{\ga}{\alp0}{\alp1}}{\eps} = \corandquot{\bswp{\ga}{\eps}}{\alp0}{\alp1}
\]
unless \(\ga=\emptyset\) and \(\eps=1\), in which case \(    \bswp{\corandquot{\emptyset}{\alp0}{\alp1}}{1} = \corandquot{\emptyset}{\alp0}{{\alp1}'}\).
    \item
    \label{item:no_of_net_spin_addables}
Let \(\al \in \scrd\) with \tbc \(\ga\).
Then \(\netspinaddables{\eps}{\al} = \netspinaddables{\eps}{\ga}\).
\end{enumerate}
\end{proposition}

\begin{proofenum}
\item
Recall that the positive half of runner \(-1\) encodes the parts congruent to \(2 \ppmod{3}\) and the positive half of runner \(1\) encodes the parts congruent to \(1 \ppmod{3}\).
Part \Cref{item:runner-swap_swaps_runners} then follows directly from \Cref{describebsw}.
\item
Let $\al=\corandquot{\ga}{\alp0}{\alp1}$. We claim that, unless \(\ga=\emptyset\) and \(\eps=1\), the largest runner of $\bswp\al\eps$ is the smallest runner of $\al$. To see this, recall that the largest runner is determined by the integer
\[
\delt\al=|\setbuild{d\in\al}{d\equiv1\ppmod3}|-|\setbuild{d\in\al}{d\equiv2\ppmod3}|;
\]
specifically, runner $1$ is largest \iff $\delt\al>0$. Given this, our claim follows from the equation
\begin{equation}\label{deltswap}
\delt{\bswp\al\eps}=1-\eps-\delt\al.
\end{equation}
This equation in turn follows from the observations that
\begin{itemize}
\item
if $d\in\N$ with $d\equiv1\ppmod3$, then $d\in\al$ \iff $d+1\in\bswp\al1$;
\item
if $d\in\N$ with $d\equiv2\ppmod3$, then $d\in\al$ \iff $d+2\in\bswp\al0$;
\item
$1\in\al$ \iff $1\notin\bswp\al0$.
\end{itemize}

Using our claim about the largest runner together with part (i), we see that the \tbq of $\bswp\al\eps$ is $(\alp0,\alp1)$ except in the case where $\ga=\emptyset$ and $\eps=1$, because the configuration of runner $0$ in the abacus is the same in $\al$ and $\bswp\al\eps$, and the configuration of the largest runner is the same up to a vertical shift.
For the exceptional case where $\ga=\emptyset$ and $\eps=1$, the configuration of the largest runner for $\bswp\al\eps$ is determined by the smallest runner for $\al$; then since the position $r$ in an abacus display is occupied \iff position $-r$ is empty, the partition determined by the smallest runner is the conjugate of the partition determined by the largest runner; thus the \tbq of $\bswp{\corandquot{\emptyset}{\alp0}{\alp1}}{1}$ is $(\alp0,{\alp1}')$.

It remains to consider \tbcs. Equation \cref{deltswap} (applied to both $\al$ and $\ga$) implies that $\delt{\bswp\al\eps}=\delt{\bswp\ga\eps}$, and hence that the \tbc of $\bswp\al\eps$ is $\bswp\ga\eps$. Now (ii) follows.
\item
Finally, since \(\netspinaddables{\eps}{\al} = |\bswp{\al}{\eps}| - |\al|\), part \Cref{item:no_of_net_spin_addables} follows from \Cref{item:3-bar-quotient-unchanged}.
\end{proofenum}

It remains only to examine the effect of \(\bswp{{-}}{\eps}\) on a \tbc, and compute \(\netspinaddables{\eps}{\thrc{r}{l}}\). The following \lcnamecref{lemma:runner-swap_on_cores} is an easy exercise.

\begin{lemma}
\label{lemma:runner-swap_on_cores}
Let \(r\in \set{1,2}\) and \(l \geq 1\). 
Let \(\eps \in \set{0,1}\).
Then:
\[
    \netspinaddables{\eps}{\thrc{r}{l}} =
    \begin{cases}
        -(2l-1) & \text{ if \(\eps = 0, r=1\);} \\
        \phantom{-|}2l+1 & \text{ if \(\eps = 0, r= 2\);} \\  
        \phantom{-|}l & \text{ if \(\eps = 1, r=1\);} \\
        -\phantom{|}l & \text{ if \(\eps = 1, r=2\);} 
    \end{cases}
\qquad\quad
\bswp{\thrc{r}{l}}{\eps} =
    \begin{cases}
        \thrc{2}{l-1} & \text{ if \(\eps = 0, r=1\);} \\
        \thrc{1}{l+1} & \text{ if \(\eps = 0, r=2\);} \\
        \thrc{2}{l} & \text{ if \(\eps = 1, r=1\);} \\
        \thrc{1}{l} & \text{ if \(\eps = 1, r=2\).} 
    \end{cases}
\]
Meanwhile, for \(l=0\), we have \(\netspinaddables{0}{\emptyset} = 1\), \(\bswp{\emptyset}{0} = \thrc{1}{1}\) and \(\netspinaddables{1}{\emptyset} = 0\), \(\bswp{\emptyset}{1} = \emptyset\).
\end{lemma}

\Cref{prop:runner-swap_action_on_spin_chars} now follows from \Cref{eq:runner-swap_action_on_spin_char_inexplicit,prop:runner-swap_swaps_runners,lemma:runner-swap_on_cores}, and hence we have established \Cref{prop:independent_of_core}.

\section{Claimed proportionality of characters}
\label{sec:if}

In this section we show that the `if' direction of our (reduced) main theorem \Cref{thm:main_reduced} holds.

\subsection{A symmetric function identity}\label{symfnsec}

We adopt some standard notation on symmetric functions; the essential reference is Macdonald's book \cite{macdonald1998symmetric}. In particular, we let $s_\la$ denote the Schur function labelled by a partition $\la$, and $P_\al$ the Schur P-function labelled by a strict partition $\al$. Let $\lan\ ,\ \ran$ denote the inner product on the space of symmetric functions for which the Schur functions are orthonormal.

Our first objective is to give a proof of the following identity.
This was stated as a corollary of our main result in \cite{fayersmcd2025spintospecht}; here we give the details. As far as we can tell, this result is new, though the case where $\tau=\emptyset$ is well-known.

\begin{proposition}\label{schurpcores}
Suppose $\si$ and $\tau$ are \tcps. Then $s_\si s_\tau=P_{\si+\tau}$.
\end{proposition}

The proof of \cref{schurpcores} relies on results concerning the decomposition numbers for $\hsss n$ in characteristic~$2$. We refer the reader to \cite[Section 5]{fayers20spin2alt} and the references therein for a full account, but here we summarise the essential points we need. For the current subsection only we use the notation $\bar{\ }$ for $2$-modular Brauer characters, rather than $3$-modular.

Take a non-negative integer $w$, and a \tcp $\nu=(c,c-1,\dots,1)$ with $c\gs w-1$. Let $\hat\nu$ be the partition $(2c-1,2c-5,2c-9,\dots)$, and let $n=|\nu|+2w$. Then there is a $2$-block $B$ of $\hsss n$ in which the ordinary irreducible characters are:
\begin{itemize}
\item
the inflations of the $\sss n$-characters $\chi_{\nu+2\ga\sqcup\de\sqcup\de}$, for partitions $\ga,\de$ with $|\ga|+|\de|=w$;
\item
the spin characters labelled by \stps $(\hat\nu+4\alp1)\sqcup2\alp0$, where $\alp0,\alp1$ are partitions with $\alp0$ strict and with $|\alp0|+2|\alp1|=w$.
\end{itemize}
The block $B$ is called a \emph{RoCK block} (or sometimes a \emph{Rouquier block}) of $\hsss n$ of weight $w$. An advantage of working with RoCK blocks is that we have formulae for their decomposition numbers. The irreducible Brauer characters in $B$ can be written in the form $\phi(\nu+2\mu)$, where $\mu$ ranges over partitions of $w$. (In fact $\phi(\nu+2\mu)$ is the Brauer character of the inflation of the \emph{James module} $D^{\nu+2\mu}$ for $\sss n$.) Now we have the following results.  (As usual, $\de'$ denotes the \emph{conjugate} or \emph{transpose} of a partition $\de$, and $\cm\chi\phi$ denotes the composition multiplicity of an irreducible Brauer character $\phi$ in a Brauer character $\chi$.)

\needspace{3em}
\begin{proposition}\label{rockdecomp}
Take $\nu,w$ as above. Then there is an invertible square matrix $A$ with rows and columns labelled by partitions of $w$, such that:
\begin{enumerate}
\item
if $\mu,\ga,\de$ are partitions with $|\mu|=|\ga|+|\de|=w$, then
\[
\cm{\widebar{\chi_{\nu+2\ga\sqcup\de\sqcup\de}}}{\phi(\nu+2\mu)}=\sum_{\la}A_{\la\mu}\lan s_\ga s_{\de'},s_\la\ran;
\]
\item
if $\mu,\alp0$ are partitions of $w$ with $\alp0$ strict, then
\[
\cm{\widebar{\spn{\hat\nu\sqcup2\alp0}}}{\phi(\nu+2\mu)}=2^{\len{\alp0}/2}\sum_\la A_{\la\mu}\lan P_{\alp0},s_\la\ran.
\]
\end{enumerate}
\end{proposition}

\begin{proof}
The theory of \emph{adjustment matrices} developed by Geck \cite{geckbrauertrees} shows that the decomposition matrix for $\sss n$ can be obtained as the product of the decomposition matrix of the corresponding Hecke algebra of type A at $q=-1$ and a square matrix $A$ called the \emph{adjustment matrix}. The adjustment matrix is known to be unitriangular (in fact, for a RoCK block it coincides with the decomposition matrix of an appropriate Schur algebra) and hence invertible, and the decomposition numbers for RoCK blocks of Hecke algebras were found by James and Mathas \cite[Corollary 2.6]{jamesmathasq-1}; these results combine to give (i).
Meanwhile part (ii) is a special case of  \cite[Theorem 5.3]{fayers20spin2alt}.
\end{proof}

\begin{proof}[Proof of \cref{schurpcores}]
Take $\si$ and $\tau$ to be \tcps as in \cref{schurpcores}. Let $w=|\si|+|\tau|$, and let $\alp0=\si+\tau$. Then the \stp $\hat\nu\sqcup2\alp0$ is a \emph{$4$-stepped and semicongruent} partition, as defined in \cite[\S 1]{fayersmcd2025spintospecht}. From the main theorem in \cite{fayersmcd2025spintospecht}, this means that $\widebar{\spn{\hat\nu\sqcup2\alp0}}$ is proportional to the $2$-modular reduction of another character in $B$, namely the character $\chi_{\nu+2\si\sqcup\tau\sqcup\tau}$; in fact, \cite[Theorem 1.1]{fayersmcd2025spintospecht} gives
\[
\widebar{\spn{\hat\nu\sqcup2\alp0}}=2^{\len{\alp0}/2}\widebar{\chi_{\nu+2\si\sqcup\tau\sqcup\tau}}.
\]
This means in particular that
\[
\cm{\widebar{\spn{\hat\nu\sqcup2\alp0}}}{\phi(\nu+2\mu)}=2^{\len{\alp0}/2}\cm{\widebar{\chi_{\nu+2\si\sqcup\tau\sqcup\tau}}}{\phi(\nu+2\mu)}
\]
for every $\mu\in\scrp(w)$, so that (by \cref{rockdecomp})
\[
\sum_\la A_{\la\mu}\lan P_{\alp0},s_\la\ran=\sum_{\la}A_{\la\mu}\lan s_\si s_{\tau'},s_\la\ran
\]
for every $\mu$. Because $A$ is invertible and \tcps are self-conjugate, this gives
\[
\lan P_{\alp0},s_\la\ran=\lan s_\si s_{\tau},s_\la\ran
\]
for every $\la\in\scrp(w)$. Because the Schur functions span the space of symmetric functions, we deduce that $P_{\alp0}=s_\si s_{\tau}$. 
\end{proof}

\subsection{Spin RoCK blocks in characteristic \texorpdfstring{\(3\)}{3}}

Now we go back to studying $3$-modular representation theory of $\hsss n$, and describe (spin) RoCK blocks in this context. Spin RoCK blocks in odd characteristic have been studied in detail by Kleshchev--Livesey \cite{Kleshchevlivesey} and by Fayers--Kleshchev--Morotti \cite{fkm}.

Fix $w\gs0$, and let $\ga$ be a \tbc with $\len\ga\gs w$. Then the block $B$ of $\hsss{|\ga|+3w}$ with \tbc $\ga$ and \tbw $w$ is called a spin RoCK block. We shall need the results from \cite{fkm} on decomposition numbers for $B$. Recall that the strict partitions labelling characters in $B$ are the partitions $\corandquot\ga{\alp0}{\alp1}$, where $\alp0,\alp1$ are partitions with $\alp0$ strict and with $|\alp0|+|\alp1|=w$.
In fact, because $B$ is a RoCK block, it is easy to describe these partitions explicitly: the conditions of \Cref{lemma:weakly_3-separated-form} are met and so
\[
\corandquot\ga{\alp0}{\alp1}=(\ga+3\alp1)\sqcup3\alp0.
\]
The irreducible Brauer characters are labelled by a different set of partitions. To make things simpler, we assume that $w$ plus the number of $1$-nodes in $\ga$ is an even number. (This ensures that every irreducible Brauer character is self-associate. If the condition is not met, then we replace $\ga$ with a suitable larger \tbc.) Then the irreducible Brauer characters can be naturally labelled $\phi(\ga\sqcup3\mu)$, where $\mu$ ranges over all partitions (not just strict partitions) of $w$. (So -- unlike in the case of $\sss n$ -- the set of labels for the irreducible Brauer characters is not a subset of the set of labels for the ordinary irreducible characters.)

Now we can describe the decomposition numbers for $B$. As with the results for characteristic $2$ in \cref{symfnsec}, these are known up to an adjustment matrix. The next \lcnamecref{rockdecomp3} follows from \cite[Corollary 6.2 and Theorem~6.17]{fkm}.

\begin{proposition}\label{rockdecomp3}
Take $\ga,w$ as above. Then there is an invertible square matrix $A$ with rows and columns labelled by partitions of $w$, such that if $\mu,\alp0,\alp1$ are partitions with $\alp0$ strict and with $|\mu|=|\alp0|+|\alp1|=w$, then
\[
\cm{\widebar{\spn{(\ga+3\alp1)\sqcup3\alp0}}}{\phi(\ga\sqcup3\mu)}=2^{\len{\alp0}/2}\sum_{\la,\ze}A_{\la\mu}\lan P_{\alp0},s_\ze\ran\lan s_\la,s_\ze s_{{\alp1}'}\ran,
\]
\end{proposition}

Now we can prove the main result of this section, which is the `if' part of \cref{thm:main_reduced}.

\begin{proposition}
\label{prop:if}
Suppose \(\rho\), \(\sigma\) and \(\tau\) are \tcps, and let
\begin{align*}
\alp0 &= \rho + \sigma, & \bep0 &= \rho + \tau, \\
\alp1 &= \tau,          &     \bep1 &= \sigma.
\end{align*}
If $\ga$ is any \tbc, then
\[
\br{\spn{\corandquot\ga{\alp0}{\alp1}}}=2^{(\len{\rho+\si}-\len{\rho+\tau})/2}\br{\spn{\corandquot\ga{\bep0}{\bep1}}}.
\]
\end{proposition}

\begin{proof}
Let $w=|\rho|+|\si|+|\tau|$. \cref{prop:independent_of_core} shows that if the result is true for some $\ga$ then it is true for every $\ga$, so we choose $\ga$ such that $\len\ga\gs w$ and such that $w$ plus the number of $1$-nodes of $\ga$ is even. We let $B$ be the RoCK block with \tbc $\ga$ and \tbw $w$.
Then the desired result is
\[
\cm{\widebar{\spn{(\ga+3\alp1)\sqcup3\alp0}}}{\phi(\ga\sqcup3\mu)}=2^{(\len{\alp0}-\len{\bep0})/2}\cm{\widebar{\spn{(\ga+3\bep1)\sqcup3\bep0}}}{\phi(\ga\sqcup3\mu)}
\]
for all $\mu\in\scrp(w)$.
Applying \cref{rockdecomp3} and using that the matrix $A$ is invertible, this is equivalent to
\begin{align}
\sum_\ze\lan P_{\alp0},s_\ze\ran\lan s_\la,s_\ze s_{{\alp1}'}\ran&=\sum_\ze\lan P_{\bep0},s_\ze\ran\lan s_\la,s_\ze s_{{\bep1}'}\ran
\nonumber
\\
\intertext{for all $\la\in\scrp(w)$.
Since the Schur functions are linearly independent, this is equivalent to}
 \sum_{\ze,\la}\lan P_{\alp0},s_\ze\ran \lan s_\la,s_\ze s_{{\alp1}'}\ran s_\la
&=
\sum_{\ze,\la}\lan P_{\bep0},s_\ze\ran \lan s_\la,s_\ze s_{{\bep1}'}\ran s_\la.
\label{eq:symfunctionsequality}
\end{align}
Using the orthonormality of the Schur functions with respect to \(\langle\ ,\ \rangle\),  the left-hand side of \Cref{eq:symfunctionsequality} is
\[
\sum_{\ze,\la}\lan P_{\alp0},s_\ze\ran \lan s_\la,s_\ze s_{{\alp1}'}\ran s_\la
    = \sum_\ze\lan P_{\alp0},s_\ze\ran s_\ze s_{{\alp1}'} \\
    = P_{\alp0}s_{{\alp1}'}
\]
and likewise the right-hand side is \(P_{\bep0}s_{{\bep1}'}\).
Thus both sides are equal to \(s_\rho s_\si s_\tau\) by \cref{schurpcores}, and the \lcnamecref{prop:if} follows.
\end{proof}

\section{Classification of proportional characters}\label{onlyifsec}

It remains to show the `only if' direction of our (reduced) main theorem \Cref{thm:main_reduced}, stated as \Cref{prop:only_if} below.

\begin{proposition}
\label{prop:only_if}
Suppose \(\al\) and \(\be\) are distinct \stps, and that the characters \(\spn{\al}\) and \(\spn{\be}\) of \(\hsss n\) are proportional in characteristic \(3\). Then there exist \tcps \(\rho\), \(\sigma\) and \(\tau\) such that
\begin{align*}
\alp0 &= \rho + \sigma, & \bep0 &= \rho + \tau, \\
\alp1 &= \tau,          &     \bep1 &= \sigma.
\end{align*}
\end{proposition}

Throughout \cref{onlyifsec}, we assume that \((\alp0,\alp1), (\bep0,\bep1) \in \scrd \times \scrp\) are distinct \tbqs, and that \(\spn{\corandquot{\ga}{\alp0}{\alp1}}\) and \(\spn{\corandquot{\ga}{\bep0}{\bep1}}\) are proportional in characteristic $3$ for some (equivalently, any) \tbc \(\ga\). This implies in particular that (for any $\ga$) the \stps $\corandquot{\ga}{\alp0}{\alp1}$ and $\corandquot{\ga}{\bep0}{\bep1}$ have the same size. Since they have the same \tbc, they must also have the same \tbw; that is \(|\alp0|+|\alp1| = |\bep0|+|\bep1|\). Write \(w\) for this common \tbw.

Our proof of \cref{prop:only_if} is by induction on the \tbw \(w\). The claim is trivially true if \(w \in \set{0,1}\) (since all pairs of \tbqs of these weights are of the required form).
So suppose \(w \geq 2\).

\subsection{Bar removal}

We first deduce a constraint on the \tbqs by placing the quotients on a sufficiently large \tbc and considering the removal of a large odd bar.
Given a \stp $\al$, we write $\rowrem\al = (\al_2,\al_3,\dots)$ for the partition  obtained by removing the first row.

\begin{lemma}
\label{lemma:bar_removal_calulation}
Let \(l \in \N\) satisfy \(l \gs\max\{\len{\alp1},\len{\bep1}\}\) and \(l \equiv \max\{\alp0_1+\alp1_1, \bep0_1+\bep1_1\}+1 \ppmod{2}\).
Let \(\ga = \thrc{1}{l}\) and let \(\al = \corandquot{\ga}{\alp0}{\alp1}\) and \(\be = \corandquot{\ga}{\bep0}{\bep1}\). Then \(\al\) and \(\be\) each have a unique largest bar of length congruent to \(1\) modulo \(6\), and these bars are of equal length. Furthermore, if \(B_\al\) and \(B_\be\) denote these bars, then
\[
\al \setminus B_\al
    = \corandquot{\rowrem{\ga}}{\rowrem{\alp0}}{\rowrem{\alp1}} \qquad\text{and}\qquad
\be \setminus B_\be
    = \corandquot{\rowrem{\ga}}{\rowrem{\bep0}}{\rowrem{\bep1}}.
\]
\end{lemma}

\begin{proof}
The choice of $l$ together with \cref{lemma:weakly_3-separated-form}\ref{weaksep} implies that
\(\al = (\ga + 3\alp1) \sqcup 3\alp0\)
and \(\be = (\ga + 3\bep1) \sqcup 3\bep0\).
By \cref{lemma:weakly_3-separated-form}\ref{largebar}, \(\al\) and \(\be\) have unique largest bars of length congruent to \(1 \ppmod{3}\), and each of these bars is given by the sum of the largest part divisible by $3$ and the largest part not divisible by $3$:
\begin{align*}
    |B_\al| &= (\ga + 3\alp1)_1 + 3\alp0_1 = 3(l-1 + \alp0_1 + \alp1_1) + 1,
\end{align*}
and likewise for \(\beta\).

Without loss of generality, suppose
\(
    \alp0_1 + \alp1_1 \geq \bep0_1+\bep1_1
\).
Then our choice of \(l\) specifies \(l-1 \equiv \alp0_1+\alp1_1 \ppmod{2}\), so that \(|B_\al| \equiv 1 \ppmod{6}\), and hence \(|B_\al|\) is odd. Then \Cref{prop:weight_and_remove_bars_from_prop_pair} shows that \(\be\) must also have a bar of this length, and by maximality and uniqueness of \(B_\be\), it must be \(B_\be\).

The description of these bars as the sums of the first part and the first part divisible by $3$ also tells us that removing them yields the following partitions:
\begin{align*}
    \al \setminus B_\al &= (\rowrem{\ga} + 3\rowrem{\alp1}) \sqcup 3\rowrem{\alp0}, \\
    \be \setminus B_\be &= (\rowrem{\ga} + 3\rowrem{\bep1}) \sqcup 3\rowrem{\bep0}.
\end{align*}
Our lower bound on \(l\) easily gives \(\len{\rowrem{\gamma}} = l-1 \geq \max\set{\len{\rowrem{\alp1}},\len{\rowrem{\bep1}}}\). Thus \(\corandquot{\rowrem{\ga}}{\rowrem{\alp0}}{\rowrem{\alp1}}\) and \(\corandquot{\rowrem{\ga}}{\rowrem{\bep0}}{\rowrem{\bep1}}\) are precisely the \stps written above, by \cref{lemma:weakly_3-separated-form}\ref{weaksep}.
\end{proof}

\begin{samepage}
\begin{proposition}
\label{prop:bar_removal}
Either:
\begin{enumerate}
    \item
    \label{item:bar_removal_equal}
    \(\rowrem{\alp0} = \rowrem{\bep0}\) and \(\rowrem{\alp1}=\rowrem{\bep1}\); or
    \item 
    \label{item:bar_removal_distinct}
    there exist \tcps \(\rho\), \(\sigma\) and \(\tau\) such that
\begin{align*}
\rowrem{\alp0} &= \rho + \sigma, & \rowrem{\bep0} &= \rho + \tau, \\
\rowrem{\alp1} &= \tau,          &     \rowrem{\bep1} &= \sigma.
\end{align*}
\end{enumerate}
\end{proposition}
\end{samepage}

\begin{proof}
Choose $l$ as in \cref{lemma:bar_removal_calulation}, and let $\al,\be,B_\al,B_\be$ be as in the proof of that \lcnamecref{lemma:bar_removal_calulation}. By \Cref{prop:weight_and_remove_bars_from_prop_pair} the Brauer characters \(\br{\spn{\al \setminus B_\al}}\) and \(\br{\spn{\be \setminus B_\be}}\) are proportional. So by the inductive hypothesis, the \tbqs $(\rowrem{\alp0},\rowrem{\alp1})$ and $(\rowrem{\bep0},\rowrem{\bep1})$ are either equal or of the form specified in \Cref{prop:only_if}, yielding \Cref{item:bar_removal_equal} or \Cref{item:bar_removal_distinct} respectively.
\end{proof}

\subsection{\texorpdfstring{\(0\)}{0}-restriction}

We next deduce a constraint on the \tbqs by considering the removal of all \esprms0, after placing the quotients on a \tbc of precise size.
Given a \stp $\al$, we write $\colrem\al = (\al_1-1,\al_2-1,\dots,\al_{\len{\al}}-1,0,\ldots)$ for the partition  obtained by removing the first column.

\begin{lemma}
\label{lemma:restriction_calc}
Let \(l = \max\set{\alp0_1 + \len{\alp1}, \bep0_1 + \len{\bep1}}\).
Let \(\ga = \thrc{1}{l}\), and let \(\al = \corandquot{\ga}{\alp0}{\alp1}\) and \(\be = \corandquot{\ga}{\bep0}{\bep1}\).
Then
\begin{align*}
 \al^{-0} &= \corandquot{\thrc2l}{\rowrem{\alp0}}{\colrem{\alp1}}, \\
  \be^{-0} &= \corandquot{\thrc2l}{\rowrem{\bep0}}{\colrem{\bep1}}.
\end{align*}
\end{lemma}

\begin{proof}
Without loss of generality, suppose
\begin{equation}
\label{eq:inequality_for_0_restriction}
    \alp0_1 + \len{\alp1} \geq \bep0 + \len{\bep1}
\end{equation}
so that \(l = \alp0_1 + \len{\alp1}\). Observe that in the \abd for \(\al\), the last bead on runner \(0\) appears in the same row as the first gap on runner \(1\), while there are no beads in the positive half of runner \(-1\). Thus \(\al\) has exactly one \espam0 (addable to the largest part congruent to \(0 \ppmod{3}\)).

Now suppose that the inequality \Cref{eq:inequality_for_0_restriction} is strict. Then in the \abd for \(\be\), the last bead on runner \(0\) appears in an earlier row than the first gap on runner \(1\), while there are no beads in the positive half of runner \(-1\). This means that $\be$ has no \espams0. This contradicts \Cref{prop:restriction_gives_prop_pair}. 
So in fact equality holds in \Cref{eq:inequality_for_0_restriction}, and for both \(\al\) and \(\be\) the last bead on runner \(0\) is in the same row as the first gap on runner \(1\).

We now compute the \stp $\al^{-0}$ obtained by removing all \esprms0 from \(\al\). This is achieved by shifting all beads in the \abd as far left as possible (\Cref{lemma:node_removal_bead_move_correspondence}). See \Cref{fig:0restriction_calculation} for an example of this computation.

Observe that in the \abd for \(\al\), every bead on runner \(0\) has a bead directly to its right, except the last bead on runner \(0\) which has the first gap on runner \(1\) directly to its right. Therefore, shifting all beads left has the effect of swapping runners \(1\) and \(-1\), and then moving the last bead on runner $0$ to runner $-1$. This means in particular that in $\al^{-0}$ there are no beads in the positive half of runner $1$, and $l$ beads in the positive half of runner $-1$. This implies that the \tbc of $\al^{-0}$ is $\ga_{2,l}$, and that the \(1\)st component of the \tbq of $\al^{-0}$ is determined by runner $-1$. Moreover, moving the last bead from runner \(0\) deletes the largest part of the \(0\)th component of the quotient, while filling in the first gap on runner \(-1\) decreases by \(1\) the size of every part of the \(1\)st component of the quotient. So the \tbq of  \(\al^{-0}\) is as claimed. Exactly the same argument applies for $\be$.
\end{proof}

\begin{figure}[ht]
\[
\sabacus(1,%
vvv,%
nxA,
nnA,
nbA,
nbA,
nnA,
nbA,
nnA,
nbn,
nnn,
nnA,
nnn,
nnA,
nnn,
nnn,
nnA,
nnA,
nnn,
vvv)
\quad \leadsto \quad
\sabacus(1,%
vvv,%
Axn,
Ann,
Abn,
Abn,
Ann,
Abn,
Ann,
bnn,
nnn,
Ann,
nnn,
Ann,
nnn,
nnn,
Ann,
Ann,
nnn,
vvv)
\]
%
\caption{An example to illustrate the computation of \(\al^{-0}\) as described in \Cref{lemma:restriction_calc} (with only nonnegative rows depicted in the abacus displays; the other rows can be deduced from these).
In this example \(\alp0 = (7,5,3,2)\) and \(\alp1=(5,5,3,2)\), 
and hence \(l =11\). The \(3\)-bar-\abd for \(\al = \corandquot{\thrc{1}{l}}{\alp0}{\alp1}\) is drawn on the left (note the key property that the first row containing a gap on runner \(1\) is the unique row containing a bead on runner \(0\) but not on runner \(1\)). The \(3\)-bar-\abd for \(\al^{-0}\), obtained by shifting all beads as far left as possible, is drawn on the right.
The beads originally on runner \(1\) in the \abd for \(\al^{-0}\) have been coloured blue.
We see that \((\al^{-0})^{(0)} = (5,3,2) = \rowrem{(7,5,3,2)}\) and \((\al^{-0})^{(1)} = (4,4,2,1) = \colrem{(5,5,3,2)}\).
}
\label{fig:0restriction_calculation}
\end{figure}

\begin{samepage}
\begin{proposition}
\label{prop:0-restriction}
Either:
\begin{enumerate}
    \item
    \label{item:0-restriction_equal}
    \(\rowrem{\alp0} = \rowrem{\bep0}\) and \(\colrem{\alp1}=\colrem{\bep1}\); or
    \item 
    \label{item:0-restriction_distinct}
    there exist \tcps \(\rho\), \(\sigma\) and \(\tau\) such that
\begin{align*}
\rowrem{\alp0} &= \rho + \sigma, & \rowrem{\bep0} &= \rho + \tau, \\
\colrem{\alp1} &= \tau,          &     \colrem{\bep1} &= \sigma.
\end{align*}
\end{enumerate}
\end{proposition}
\end{samepage}

\begin{proof}
Let $\al$ and $\be$ be as in \Cref{prop:0-restriction}. By \Cref{prop:restriction_gives_prop_pair} the Brauer characters \(\br{\spn{\al^{-0}}}\) and \(\br{\spn{\be^{-0}}}\) are also proportional. So by the inductive hypothesis, the \tbqs $(\rowrem{\alp0},\colrem{\alp1})$ and $(\rowrem{\bep0},\colrem{\bep1})$ are either equal or of the form specified in \cref{prop:only_if}, yielding \Cref{item:0-restriction_equal} or \Cref{item:0-restriction_distinct} respectively.
\end{proof}

\subsection{Inductive step}

To complete the proof of \cref{prop:only_if} by induction, we consider each of the four combinations of possibilities arising from \Cref{prop:bar_removal,prop:0-restriction}, which correspond to whether or not the pairs of partitions found to label proportional characters using the inductive hypothesis are in fact equal partitions.

Recall that we are assuming that \((\alp0, \alp1), (\bep0,\bep1)\in\scrd\times\scrp\) are distinct, with $|\alp0|+|\alp1|=|\bep0|+|\bep1|$. Our aim is to show that
\[
(\alp0,\alp1,\bep0,\bep1)=(\rho+\si,\tau,\rho+\tau,\si)
\]
for some \tcps $\rho,\si,\tau$.

\subsubsection{The case that \texorpdfstring{\labelcref{prop:bar_removal}\Cref{item:bar_removal_equal}}{6.3(i)} and \texorpdfstring{\labelcref{prop:0-restriction}\Cref{item:0-restriction_equal}}{6.5(i)} hold}
\label{case:two-equal}

Suppose \(\rowrem{\alp0} = \rowrem{\bep0}\), \(\rowrem{\alp1} = \rowrem{\bep1}\) and \(\colrem{\alp1} = \colrem{\bep1}\).

Having \(\rowrem{\alp1} = \rowrem{\bep1}\) and \(\colrem{\alp1} = \colrem{\bep1}\) implies that either \(\alp1 = \bep1\) or \(\set{\alp1,\bep1} = \set{\emptyset, (1)}\). If \(\alp1 = \bep1\), then we deduce \(|\alp0| = |\bep0|\), which together with \(\rowrem{\alp0} = \rowrem{\bep0}\) implies \(\alp0 = \bep0\), contradicting the assumption that \((\alp0, \alp1), (\bep0,\bep1)\) are distinct.

So we suppose (without loss of generality) that \(\alp1 = \emptyset\) and \(\bep1 = (1)\). Then, using \(\rowrem{\alp0} = \rowrem{\bep0}\), we obtain \(\alp0 = \bep0 + (1)\).

If \(\bep0\) is a \tcp, then \((\alp0,\alp1)\) and \((\bep0,\bep1)\) are of the required form (with $(\rho,\si,\tau)=(\bep0,(1),\emptyset)$). So suppose \(\bep0\) is not a \tcp; let \(l \geq 1 = \max\{\len{\alp1},\len{\bep1}\}\) be minimal such that \(l \not\in \bep0\). Then $l<\bep0_1$ because $\bep0$ is not a \tcp, and therefore $l\not\in\alp0$ also.

Let \(\ga=\thrc{1}{l}\) and let \(\al = \corandquot{\ga}{\alp0}{\alp1}\) and \(\be = \corandquot{\ga}{\bep0}{\bep1}\). Then \Cref{lemma:weakly_3-separated-form}\ref{weaksep} gives
\begin{align*}
    \al &= (\ga + 3\alp1) \sqcup 3\alp0 = \ga \sqcup 3(\bep0 + (1)), \\
    \be &= (\ga + 3\bep1) \sqcup 3\bep0 = (\ga+(3)) \sqcup3\bep0.
\end{align*}
We will derive a contradiction by using the \abds of $\al$ and $\be$ to compare the number of \esprms0 of $\al$ and $\be$. The choice of $l$ means that, in both \abds, the positive half of runner \(-1\) is empty, and runner \(0\) has a bead in every row from $1$ to \(l-1\) and a gap in row \(l\).
Meanwhile \Cref{lemma:pos_of_gap_and_no_of_beads} 
shows that the first gap on runner \(1\) is in row \(l\) in \(\al\), and in row \(l-1\) in \(\be\).
Furthermore, since \(\bep1 = (1)\), there is a bead on runner \(1\) in row \(l\), in the \abd for \(\be\).
Rows \(0\) to \(l\) are illustrated in \Cref{fig:lrows_casetwoequal}.

\begin{figure}
\makeatletter
\centering
\begin{tikzpicture}[bend angle=0, scale=\abas*1.4,baseline={([yshift=-.8ex]current bounding box.center)}]
\@bacus
vvv,%
nxb,
nbb,
nbb,
vvv,%
nbb,
nbb,
nnn,
vvv.%
\coordinate[at={(-\abah*1.2,-\abav*0.5)}] (lhsz);
\node[abbelow, at={(\abah,-\abav*9)}] {\(\al\)};
\useasboundingbox (current bounding box);
\node[ablhs, at={($(lhsz)+(0,-\abav*5)$)}] {row \(l{-}1\)};
\node[ablhs, at={($(lhsz)+(0,-\abav*6)$)}] {row \(l\phantom{-1}\)};
\end{tikzpicture}
\quad\;
\begin{tikzpicture}[bend angle=0, scale=\abas*1.4,baseline={([yshift=-.8ex]current bounding box.center)}
]
\@bacus
vvv,%
nxb,
nbb,
nbb,
vvv,%
nbb,
nbn,
nnb,
vvv.%
\node[abbelow, at={(\abah,-\abav*9)}] {\(\be\)};
\end{tikzpicture}
\caption{Rows \(0\) to \(l\) of the \abds for \(\al\) and \(\be\) considered in \Cref{case:two-equal}.}
\label{fig:lrows_casetwoequal}
\makeatother
\end{figure}

We now compare numbers of \esprms0. In rows \(0\) to \(l-2\), the \abds are identical, so there is no difference in number of \esprms0. In rows \(l+1\) and below, there are beads only on runner \(0\), and an equal number of such beads on the two displays (namely, the number of parts of $\bep0$ greater than~$l$), so there is no difference in number of \esprms0.
In rows \(l-1\) and \(l\), \(\al\) has two \esprms0 (both on row \(l-1\)) while \(\be\) has three \esprms0 (one on row \(l-1\), two on row \(l\)).
Overall, then, \(\be\) has one more \esprm0 than \(\al\).
This contradicts \Cref{prop:restriction_gives_prop_pair}.

\subsubsection{The case that
\texorpdfstring{\labelcref{prop:bar_removal}\Cref{item:bar_removal_equal}}{6.3(i)}
and
\texorpdfstring{\labelcref{prop:0-restriction}\Cref{item:0-restriction_distinct}}{6.5(ii)}
hold}

Suppose \(\rowrem{\alp0} = \rowrem{\bep0}\), \(\rowrem{\alp1} = \rowrem{\bep1}\), and that     there exist \tcps \(\rho\), \(\sigma\) and \(\tau\) such that
\begin{align*}
\rowrem{\alp0} &= \rho + \sigma, & \rowrem{\bep0} &= \rho + \tau, \\
\colrem{\alp1} &= \tau,          &     \colrem{\bep1} &= \sigma.
\end{align*}
We deduce \(\rho+\sigma = \rho+\tau\) and hence \(\sigma = \tau\).
Then \(\colrem{\alp1} = \colrem{\bep1}\), so we are in case \Cref{case:two-equal}.

\subsubsection{The case that
\texorpdfstring{\labelcref{prop:bar_removal}\Cref{item:bar_removal_distinct}}{6.3(ii)}
and
\texorpdfstring{\labelcref{prop:0-restriction}\Cref{item:0-restriction_equal}}{6.5(i)}
hold
}

Suppose \(\rowrem{\alp0} = \rowrem{\bep0}\), \(\colrem{\alp1} = \colrem{\bep1}\), and that     there exist \tcps \(\rho\), \(\sigma\) and \(\tau\) such that
\begin{align*}
\rowrem{\alp0} &= \rho + \sigma, & \rowrem{\bep0} &= \rho + \tau, \\
\rowrem{\alp1} &= \tau,          &     \rowrem{\bep1} &= \sigma.
\end{align*}
We deduce \(\rho+\sigma = \rho+\tau\) and hence \(\sigma=\tau\).
Then \(\rowrem{\alp1} = \rowrem{\bep1}\), so we are in case \Cref{case:two-equal}.

\subsubsection{The case that
\texorpdfstring{\labelcref{prop:bar_removal}\Cref{item:bar_removal_distinct}}{6.3(ii)}
and
\texorpdfstring{\labelcref{prop:0-restriction}\Cref{item:0-restriction_distinct}}{6.5(ii)}
hold}
\label{subsec:neither}

This final case requires considerably more work, and will occupy the remainder of this section. Suppose there exist integers \(a,b,c,a',b',c' \geq 0\) such that
\begin{align*}
\rowrem{\alp0} &= \twoc{a} + \twoc{b} = \twoc{a'}+\twoc{b'},
    & \rowrem{\bep0} &= \twoc{a} + \twoc{c} = \twoc{a'}+\twoc{c'}, \\
\rowrem{\alp1} &= \twoc{c},          &     \rowrem{\bep1} &= \twoc{b}, \\
\colrem{\alp1} &= \twoc{c'},          &     \colrem{\bep1} &= \twoc{b'}.
\end{align*}
The equality \(\twoc{a} + \twoc{b} = \twoc{a'}+\twoc{b'}\) implies \(\set{a,b} = \set{a',b'}\), and likewise \(\set{a,c} = \set{a',c'}\). Observe that if \(a \neq a'\), then \(a=b'=c'\) and \(a'=b=c\), and hence \(\rowrem{\alp1}=\rowrem{\bep1}\) and \(\colrem{\alp1}=\colrem{\bep1}\), and we are in case \Cref{case:two-equal}.
So suppose \(a=a'\), and hence \(b=b'\) and \(c=c'\).

When \(d=-1\), we abuse notation by writing $\la+(-1)$ to mean $(\la_1-1,\la_2,\la_3,\dots)$.

\begin{lemma}
\label{lemma:reparametrise_cores}
There exist \tcps \(\rho, \sigma, \tau\) and an integer \(d \geq -1\) such that
\begin{align*}
\alp0 &= \rho + \sigma + (d), & \bep0 &= \rho + \tau + (d), \\
\alp1 &= \tau,          &     \bep1 &= \sigma.
\end{align*}
\end{lemma}

\begin{proof}
We let $\tau=\alp1$ and $\si=\bep1$, and we must show that these partitions are \tcps. By assumption \(\rowrem\tau = \colrem\tau= \twoc{c}\).
If \(c > 0\), then \(\rowrem\tau= \twoc{c}\) implies \(\tau = \twoc{c} \sqcup (x)\) for some \(x\gs c\), and then \(\colrem{\tau} = \twoc{c}\) implies
\[
\twoc c=\colrem{(\twoc{c} \sqcup (x))} = \twoc{c-1} \sqcup (x-1);
\]
thus \(x=c+1\) and hence \(\tau= \twoc{c+1}\). If \(c=0\), then \(\rowrem\tau=\colrem\tau=\emptyset\) and so \(\tau\in\set{ \emptyset, (1)}\). In any case \(\tau\) is a \tcp. Likewise \(\si\) is a \tcp.

Let \(\rho = \twoc{a+1}\).
Observe that \(\rowrem{(\rho + \sigma)} = \twoc{a}+\twoc{b} = \rowrem{\alp0}\), so that \(\alp0\) differs from \(\rho+\sigma\) only possibly in the first row, i.e.~there exists \(d \in \Z\) such that \(\alp0 = \rho + \sigma + (d)\).
Likewise there exists \(d'\in \Z\) such that \(\bep0= \rho+\tau+(d')\).
But since \(|\alp0|+|\alp1| = |\bep0|+|\bep1|\), we see that \(d=d'\).
Finally, if \(d \leq -2\) then \(\rho_1+\sigma_1 +d \leq \rho_2+\sigma_2\), and so \(\alp0\) is not a strict partition; thus we have \(d \geq -1\).
\end{proof}

Let \(\rho,\sigma,\tau,d\) be as in \Cref{lemma:reparametrise_cores},
and set \(r = \rho_1\), \(s = \sigma_1\), 
\(t = \tau_1\).
If \(s=t\) (i.e.~if \(\sigma=\tau\)) then \((\alp0,\alp1) = (\bep0,\bep1)\), a contradiction.
So, without loss of generality, suppose \(s < t\).

Observe that \((\alp0,\alp1)\), \((\bep0,\bep1)\) are of the required form if \(d=0\), if \(d=-1\) and \(r=1\), or if \(d=1\) and \(r=0\).
Meanwhile if \(d=-1\) and \(r=0\) then \(\alp0\) and \(\bep0\) are not \stps, a contradiction.
It remains to suppose \(d\geq 2\), or \(d=1\) and \(r \geq 1\), or \(d=-1\) and \(r \geq 2\), and reach a contradiction.
Note that in any of these cases \(r+d \geq 1\).

Let \(\al = \corandquot{\emptyset}{\alp0}{\alp1}\) and \(\be = \corandquot{\emptyset}{\bep0}{\bep1}\).
We aim to show that (in most cases) \(\al^{-0}\) and \(\be^{-0}\) are distinct, strictly smaller than \(\al\) and \(\be\), and not of the form described in \Cref{prop:only_if}, contradicting the inductive hypothesis.
We first describe \(\al\) and \(\be\) more explicitly.

\begin{lemma}
\label{lemma:describe_al_and_be}
In the abacus display for \(\al\):
\begin{enumerate}[(i)]
    \item the largest runner is runner \(-1\);
    \item in rows \(-t\) and \(t\), there is a bead on runner \(-1\);
    \item between rows \(-t\) and \(t\) (inclusive), there is alternately a bead on runner \(-1\) and runner \(1\);
    \item in rows \(i > t\), there are no beads on runners \(-1\) and \(1\).
\end{enumerate}
The same statements with \(s\) in place of \(t\) hold for \(\be\).
The parts of \(\al\) and \(\be\) are given by
\begin{align*}
    \al &= 3(\rho+\sigma+(d)) \sqcup
    \begin{cases}
        (3t-1, 3t-2, 3t-7, 3t-8, \ldots, 5,4) & \text{ if \(t \equiv 0 \ppmod{2}\),} \\
        (3t-1, 3t-2, 3t-7, 3t-8, \ldots, 2,1) & \text{ if \(t \equiv 1 \ppmod{2}\),}
    \end{cases} \\
    \be &= 3(\rho+\tau+(d)) \sqcup
    \begin{cases}
        (3s-1, 3s-2, 3s-7, 3s-8, \ldots, 5,4) & \text{ if \(s \equiv 0 \ppmod{2}\),} \\
        (3s-1, 3s-2, 3s-7, 3s-8, \ldots, 2,1) & \text{ if \(s \equiv 1 \ppmod{2}\).}
    \end{cases}
\end{align*}
In particular,
\begin{align*}
    \al_1 &= \max\set{3t-1, 3r+3s+3d}, & \be_1 &= 3r+3t+3d.
\end{align*}
\end{lemma}

\newcommand{\absize}{1.2}
\makeatletter
\begin{figure}
\begin{subfigure}{0.32\linewidth}
\centering
\begin{tikzpicture}[bend angle=0, scale=\abas*\absize,baseline={([yshift=-.8ex]current bounding box.center)}]
\@bacus
vvv,%
bbb,
bbb,
bbb,
bnb,
bbb,
bbn,
nbb,
bbn,
nnb,
bxn,
nbb,
bnn,
nnb,
bnn,
nnn,
nbn,
nnn,
nnn,
nnn,
vvv.%
\coordinate[at={(-\abah*0.8,-\abav*9.5)}] (lhsz);
\node[ablhs, at={($(lhsz)+(0,-\abav*4)$)}] {\(t\)};
\node[ablhs, at={($(lhsz)+(0,\abav*4)$)}] {\(-t\)};
\node[abbelow, at={(\abah,-\abav*21)}] {\(\al\)};
\end{tikzpicture}
\quad\;
\begin{tikzpicture}[bend angle=0, scale=\abas*\absize,baseline={([yshift=-.8ex]current bounding box.center)}]
\@bacus
vvv,%
bbb,
bnb,
bbb,
bbb,
bbb,
bbb,
bnb,
bnn,
nnb,
bxn,
nbb,
bbn,
nbn,
nnn,
nnn,
nnn,
nnn,
nbn,
nnn,
vvv.%
\coordinate[at={(-\abah*0.8,-\abav*9.5)}] (lhsz);
\node[ablhs, at={($(lhsz)+(0,-\abav*2)$)}] {\(s\)};
\node[ablhs, at={($(lhsz)+(0,\abav*2)$)}] {\(-s\)};
\node[abbelow, at={(\abah,-\abav*21)}] {\(\be\)};
\end{tikzpicture}
\caption{\(r=1\), \(s=2\), \(t=4\), \(d=3\)}
\label{subfig:long_part_in_be-0}
\end{subfigure}
\begin{subfigure}{0.32\linewidth}
\centering
\begin{tikzpicture}[bend angle=0, scale=\abas*\absize,baseline={([yshift=-.8ex]current bounding box.center)}]
\@bacus
vvv,%
bbb,
bbb,
bbb,
bnb,
bnb,
bbn,
nbb,
bbn,
nnb,
bxn,
nbb,
bnn,
nbb,
bnn,
nbn,
nbn,
nnn,
nnn,
nnn,
vvv.%
\coordinate[at={(-\abah*0.8,-\abav*9.5)}] (lhsz);
\node[ablhs, at={($(lhsz)+(0,-\abav*4)$)}] {\(t\)};
\node[ablhs, at={($(lhsz)+(0,\abav*4)$)}] {\(-t\)};
\node[abbelow, at={(\abah,-\abav*21)}] {\(\al\)};
\end{tikzpicture}
\quad\;
\begin{tikzpicture}[bend angle=0, scale=\abas*\absize,baseline={([yshift=-.8ex]current bounding box.center)}]
\@bacus
vvv,%
bbb,
bbb,
bnb,
bnb,
bbb,
bbb,
bnn,
nbb,
bnn,
nxb,
bbn,
nnb,
bbn,
nnn,
nnn,
nbn,
nbn,
nnn,
nnn,
vvv.%
\coordinate[at={(-\abah*0.8,-\abav*9.5)}] (lhsz);
\node[ablhs, at={($(lhsz)+(0,-\abav*3)$)}] {\(s\)};
\node[ablhs, at={($(lhsz)+(0,\abav*3)$)}] {\(-s\)};
\node[abbelow, at={(\abah,-\abav*21)}] {\(\be\)};
\end{tikzpicture}
\caption{\(r=4\), \(s=3\), \(t=4\), \(d=-1\)}
\label{subfig:repeated_part_in_be-0}
\end{subfigure}
\begin{subfigure}{0.32\linewidth}
\centering
\begin{tikzpicture}[bend angle=0, scale=\abas*\absize,baseline={([yshift=-.8ex]current bounding box.center)}]
\@bacus
vvv,%
bbb,
bbb,
bbb,
bbb,
bbb,
bbb,
bbb,
bnn,
nnb,
bxn,
nbb,
bbn,
nnn,
nnn,
nnn,
nnn,
nnn,
nnn,
nnn,
vvv.%
\coordinate[at={(-\abah*0.8,-\abav*9.5)}] (lhsz);
\node[ablhs, at={($(lhsz)+(0,-\abav*2)$)}] {\(t\)};
\node[ablhs, at={($(lhsz)+(0,\abav*2)$)}] {\(-t\)};
\node[abbelow, at={(\abah,-\abav*21)}] {\(\al\)};
\end{tikzpicture}
\quad\;
\begin{tikzpicture}[bend angle=0, scale=\abas*\absize,baseline={([yshift=-.8ex]current bounding box.center)}]
\@bacus
vvv,%
bbb,
bbb,
bbb,
bbb,
bbb,
bbb,
bnb,
bnb,
bbn,
nxb,
bnn,
nbn,
nbn,
nnn,
nnn,
nnn,
nnn,
nnn,
nnn,
vvv.%
\coordinate[at={(-\abah*0.8,-\abav*9.5)}] (lhsz);
\node[ablhs, at={($(lhsz)+(0,-\abav*1)$)}] {\(s\)};
\node[ablhs, at={($(lhsz)+(0,\abav*1)$)}] {\(-s\)};
\node[abbelow, at={(\abah,-\abav*21)}] {\(\be\)};
\end{tikzpicture}
\caption{\(r=2\), \(s=1\), \(t=2\), \(d=-1\)}
\label{subfig:count0nodes}
\end{subfigure}

    \caption{Examples of \abds for \(\al = \corandquot{\emptyset}{\alp0}{\alp1}\) and \(\be = \corandquot{\emptyset}{\bep0}{\bep1}\) considered in \Cref{subsec:neither}.
    The indices of some rows are recorded on the left-hand side of each abacus.}
    
    \label{fig:al_and_be_in_6.3.4}

\makeatother
\end{figure}

\begin{proof}
The parts of \(\al\) and \(\be\) divisible by \(3\) are given by \(3\alp0\) and \(3\bep0\), as is the case for any strict partition. It remains to consider the parts not divisible by \(3\) and the positions of the beads on runners~\(-1\) and~\(1\).

Construct an \abd as follows (which we will deduce is the \abd for \(\al\); several examples of this construction are depicted in \Cref{fig:al_and_be_in_6.3.4}).
Place beads on runner \(-1\) in rows \(-t, -t+2, \ldots, t-2, t\) (and all rows before row \(-t\)).
The symmetry of runners~\(-1\) and~\(1\) then determines that runner~\(1\) has beads in rows \(-t+1, -t+3, \ldots, t-3, t-1\) (and all rows strictly before row \(-t\)), which, between rows~\(-t\) and~\(t\), are precisely the rows that runner~\(-1\) has a gap.
On these runners, sliding all the beads up as far as possible yields no beads in positive positions (as can be seen on runner~\(-1\) by matching the bead in row \(i > 0 \) with the gap in row \(-i+1\), or on runner~\(1\) by matching the bead in row \(i\geq0\) with the gap in row \(-i-1\)).
Thus the \abd has empty \tbc, and so the largest runner is runner~\(-1\).
Then the first component of the \tbq is \(\twoc{t}\) (since runner \(-1\) alternates between gaps and beads, with exactly \(t\) beads contributing nonzero parts).
Thus (after adding \(3\alp0\) appropriately) this is the \abd for \(\al\).
This \abd fulfils the description in the statement of the lemma, and the parts of \(\al\) not divisible by \(3\) can then be read off the abacus.

The same reasoning holds for \(\be\) with \(s\) in place of \(t\).
The value of \(\be_1\) follows using \(s < t\) and \(r+d \geq 1\) to deduce \(3s-1 < 3r+3t+3d\).
\end{proof}

\begin{lemma}
\label{lemma:al0_and_be0}
The largest part of \(\be^{-0}\) has size \(3r+3t+3d-1\), and satisfies \((\al^{-0})_1 < (\be^{-0})_1 < \be_1\).
In particular, \(\be^{-0}\) is distinct from \(\be\) and from \(\al^{-0}\).
\end{lemma}

\begin{proof}
The largest part of \(\be\), of size \(\be_1 = 3r+3t+3d\), is a part with a single \(0\)-node at its end.
This node is removable since the largest part of \(\be\) congruent to \(-1 \ppmod{3}\) is \(3s-1\) which is strictly less than \(3r+3t+3d-1\) (using \(s < t\) and \(r+d \geq 1\)).
This gives \((\be^{-0})_1 = 3r+3t+3d-1 < \be_1\).
This value satisfies \((\be^{-0})_1 > 3t-1\) (using \(r+d \geq 1\)) and \((\be^{-0})_1 > 3r+3s+3d\) (using \(t>s\)), and so \((\be^{-0})_1 > \al_1 \geq (\al^{-0})_1\).
\end{proof}

\needspace{3em}
\begin{lemma}\label{d2d1r1}
Suppose either:
\begin{enumerate}[(a)]
    \item \(d \geq 2\); or
    \item \(d=1\) and \(r \geq 1\); or
    \item \(d=-1\) and \(r \geq \max\{2, 4-(t-s)\}\).
\end{enumerate}
Then \((\be^{-0})^{(1)}\) is not a \tcp.
\end{lemma}

\begin{proof}
Consider the \abd for \(\be^{-0}\).
The largest runner 
is runner \(-1\) (because this is the case for \(\be\) by \Cref{lemma:describe_al_and_be}, and moving beads left can only decrease the quantity \(\de\) defined in \Cref{subsec:abacus}).

[(a) and (b)]
From the value of \((\be^{-0})_1\) computed in \Cref{lemma:al0_and_be0}, the lowest bead on runner~\(-1\) is in row \(r+t+d\).
The second-lowest bead on runner \(-1\) corresponds to the second-largest part of \(\be^{-0}\) congruent to \(-1 \ppmod{3}\).
This is either the largest part of \(\be\) that is already congruent to \(-1 \ppmod{3}\), namely $3s-1$, or is obtained by subtracting $1$ from the second-largest part of \(\be\) congruent to \(0 \ppmod{3}\), which is $3\max\set{r-1,0} + 3(t-1)$.
The assumption that $s<t$ implies that the second of these integers is the larger, so
the second-lowest bead on runner $-1$ of the abacus is in row \(\max\set{r-1,0}+t-1\).

The number of empty rows between these beads, then, is
\[
    (r+t+d) - (\max\set{r-1,0}+t-1) - 1
    = \begin{cases}
        d+1 & \text{ if \(r\geq 1\),} \\
        d & \text{ if \(r=0\).}
    \end{cases}
\]
Thus the largest part of \((\be^{-0})^{(1)}\) is at least \(2\) greater than the second-largest part provided \(d \geq 2\), or \(d=1\) and \(r \geq 1\).
An example of \abds for \(\al\) and \(\be\) satisfying these conditions is depicted in \Cref{subfig:long_part_in_be-0}.

[(c)]
Suppose \(d=-1\) and \(r \geq \max\{2, 4-(t-s)\}\)
(an example of \abds for \(\al\) and \(\be\) satisfying these conditions is depicted in \Cref{subfig:repeated_part_in_be-0}).
The fact that \(\alp0 = \rho+\sigma+(-1)\) is a \stp gives \(s \geq 1\), and therefore \(t \geq 2\).
The lowest bead on runner \(-1\) in an \abd for \(\be^{-0}\) is in row \(r+t+d = r+t-1\) as before.
We claim that there is also a bead on runner \(-1\) in row \(r+t-2\), but not in row \(r+t-3\).
Indeed, the bounds \(r,t \geq 2\) guarantee that the next lowest beads on runner \(0\) in an \abd for \(\be\) are in rows \(r+t-2\) and \(r+t-4\), while the lowest bead on runners \(-1\) or \(1\) occurs in row \(s\) (using \Cref{lemma:describe_al_and_be}), which is strictly above row \(r+t-3\) by the assumption \(r \geq 4-(t-s)\).
Two adjacent beads with a gap before them implies \(\be^{-0}\) has a repeated part.
\end{proof}

If one of the hypotheses of \cref{d2d1r1} is met, then \(\al^{-0}\) and \(\be^{-0}\) are not of the form described in \Cref{prop:only_if}, and we obtain our contradiction with the inductive hypothesis. The only case remaining is \(d=-1\) and \(4-(t-s) > r \geq 2\). Since we assume $t>s$, this can happen only if \(d=-1\), \(t=s+1\) and \(r=2\) (an example of \abds for \(\al\) and \(\be\) satisfying these conditions is depicted in \Cref{subfig:count0nodes}).
We find a contradiction in this last case by counting the \esprms0 of $\al$ and $\be$.

\begin{lemma}
\label{d-1_t=s+1_r=2}
Suppose \(d=-1\), \(t=s+1\) and \(r=2\).
Then the number of \esprms0 in \(\al\) is \(t\) and in \(\be\) is \(t+1\).
\end{lemma}

\begin{proof}
The largest part of \(\al\) congruent to \(0 \ppmod{3}\) is \(3s+3\), corresponding to a bead on runner \(0\) in row \(t\).
Given the description of beads in runners \(-1\) and \(1\) from \Cref{lemma:describe_al_and_be}, we then see that every bead in runner \(0\) is in a row which also contains a bead in runner \(-1\) or \(1\).
Thus the \esprms0 can be counted just by considering moves from positive positions in runner \(1\) to runner \(-1\).
If \(t\) is even, there are \(t/2\) parts congruent to \(1 \ppmod{3}\) each with two \esprms0;
if \(t\) is odd, there are \((t-1)/2\) parts congruent to \(1 \ppmod{3}\) each with two \esprms0, plus the part equal to \(1\) with one \esprm0.
In either case, the total number of \esprms0 is \(t\).

The largest two parts of \(\be\) congruent to \(0 \ppmod{3}\) are \(3t+3\) and \(3t\), corresponding to beads in rows \(s+2\) and \(s+1\).
Given the description of beads in runners \(-1\) and \(1\) from \Cref{lemma:describe_al_and_be}, these two beads can each be moved left, contributing two \esprms0, but the remaining rows containing beads in runner \(0\) also contain a bead in runner \(-1\) or \(1\).
Counting the parts congruent to \(1 \ppmod{3}\) as in the previous paragraph gives \(s\) \esprms0 from the remaining rows, for a total of \(s+2 = t+1\).
\end{proof}

This completes the proof of \Cref{prop:only_if} and hence of our main theorem.

\bibliographystyle{alpha-noxc}
\bibliography{references}

\end{document}